\documentclass[11pt, a4paper]{amsart}
\usepackage[hmargin=2.7cm,vmargin=2.7cm]{geometry}

\usepackage[mathscr]{eucal}
\usepackage{amsfonts}
\usepackage{amsmath}
\usepackage{amsthm}
\usepackage{empheq,mathtools}
\usepackage[colorlinks]{hyperref}
\hypersetup{
 colorlinks=true,
 citecolor=magenta,
 linkcolor=blue,
 urlcolor=cyan}
\usepackage{enumitem}
\usepackage{tikz}
\usetikzlibrary{matrix,arrows,decorations.pathmorphing}
\usetikzlibrary{positioning}
\usepackage{xcolor}
\usepackage{bm}
\usepackage[cal=boondox]{mathalfa}

\newtheorem{theorem}{Theorem}
\newtheorem{lemma}[theorem]{Lemma}
\newtheorem{corollary}[theorem]{Corollary}
\newtheorem{proposition}[theorem]{Proposition}
\newtheorem{definition}[theorem]{Definition}

\newtheorem{remark}[theorem]{Remark}

\makeatletter
\let\reftagform@=\tagform@
\def\tagform@#1{\maketag@@@{(\ignorespaces\textcolor{cyan}{#1}\unskip\@@italiccorr)}}
\renewcommand{\eqref}[1]{\textup{\reftagform@{\ref{#1}}}}
\makeatother
\usepackage{hyperref}
\hypersetup{colorlinks=true}

\newcommand{\co}{co}
\newcommand{\cov}{\mbox{\rm cov}\hspace{0.02in}}
\newcommand{\lip}{\mbox{\rm lip}\hspace{0.02in}}

\newcommand{\reg}{\mbox{\rm reg}\hspace{0.02in}}

\begin{document}


\title[Metric regularity, pseudo-Jacobians and inversion on Finsler manifolds]{Metric regularity, pseudo-Jacobians and global inversion theorems on Finsler manifolds}
\author{Olivia Gut\'u$^1$, Jes\'us A. Jaramillo$^2$ and \'Oscar Madiedo$^3$}

\address{$^1$ Departamento de Matem{\'a}ticas\\  Universidad de Sonora, 83000\\ Hermosillo, Sonora, Mexico}

\email{oliviagutu@mat.uson.mx}

\address{$^2$ Instituto de Matem\'atica Interdisciplinar (IMI) and Departamento de An{\'a}lisis Matem{\'a}tico\\  Universidad Complutense de Madrid\\ 28040\\ Madrid, Spain}

\email{jaramil@mat.ucm.es}

\address{$^3$ Departamento de Econom\'ia Financiera y Actuarial y Estad\'istica\\ Universidad Complutense de Madrid\\ Campus de Somosaguas,  28223 \\Pozuelo de Alarc\'on, Madrid, Spain}

\email{omadievo@ucm.es}

\thanks{Research supported in part by MICINN (Spain), grant PGC2018-097286-B-I00}

\keywords{Global invertibility; Finsler manifolds; Nonsmooth analysis}

\subjclass[2000]{49J52, 58B20, 46G05}


\maketitle


\begin{abstract}
Our aim in this paper is to study the global invertibility of a locally Lipschitz map $f:X \to Y$ between (possibly infinite-dimensional) Finsler manifolds, stressing the connections with covering properties and metric regularity of $f$. To this end, we introduce a natural notion of pseudo-Jacobian $Jf$ in this setting, as is a kind of set-valued differential object associated to $f$. By means of a suitable index, we study the relations between properties of pseudo-Jacobian $Jf$ and local metric properties of the map $f$, which lead to conditions for $f$ to be a covering map, and for $f$ to be globally invertible. In particular, we obtain a version of Hadamard integral condition in this context.
\end{abstract}


\section{Introduction}

Global invertibility of mappings is an important issue in nonlinear functional analysis. In a smooth setting, if $f:X\to Y$ is a $C^1$-map between Banach spaces satisfying that its derivative $f'(x)$ is a linear isomorphism for every $x\in X$, from the classical Inverse Function Theorem we have that $f$ is locally invertible around each point. If $f$ satisfies in addition the so-called {\it Hadamard integral condition,} that is, if
$$
\int_0^{\infty} \inf_{\Vert x \Vert \leq t} \Vert f'(x)^{-1} \Vert^{-1} \, {\rm d}t = \infty,
$$
then $f:X\to Y$ is globally invertible, and thus a global diffeomorphism from $X$ onto $Y$. We refer to Plastock \cite{Plastock} for a proof of this result.  Analogous conditions for global invertibility have been also obtained in the more general setting where $f:X\to Y$ is a $C^1$-map between Finsler manifolds, mainly in connection with covering maps, the path-lifting property, and other related topological properties. We refer to the survey of Gut\'u \cite{gutu2015} for an extensive and detailed information about global invertibility of smooth maps between  Finsler manifolds.

\

In this paper we will be interested on global invertibility in a nonsmooth setting, more precisely in the case where $f: X \to Y$ is a locally Lipschitz map between Finsler manifolds modeled on Banach spaces (in the sense of Palais \cite{palais}). When $X$ and $Y$ are finite-dimensional Finsler manifolds, this kind of problems have been considered in \cite{JaMaSa}, where in particular a version of Hadamard integral condition for global invertibility is obtained, in terms of a suitable analog of the {\it Clarke generalized Jacobian} $\partial f$. For continuous maps  $f: \mathbb R^n \to \mathbb R^n$,  Jeyakumar and Luc introduced in \cite{JL0} the more general concept of {\it approximate Jacobian matrix,} which later on was named {\it pseudo-Jacobian matrix} (see \cite{JL}). Furthermore, a global inversion theorem in terms of such matrices is given in \cite{JaMaSa1},  with a version of the Hadamard integral condition in this context. These results have been extended in \cite{GHP} to the case where $f: X \to Y$ is a continuous map between finite-dimensional Riemannian manifolds, which admits an analog of pseudo-Jacobian in this setting.

\

If $f:X\to Y$ is a nonsmooth map between infinite-dimensional Banach spaces, the problem of local invertibility of $f$ is more delicate. Assuming that $f$ is a local homeomorphism, F. John obtained in \cite{john} a global inversion theorem using a suitable version of the Hadamard integral condition in terms of the {\it lower scalar  Dini derivative} of $f$, which is defined for  each $x\in X$ as
$$
D^{-}_x f =\liminf_{z\to x} \frac{\|f(z) - f(x)\|}{\|z - x\|}.
$$
Namely, he proved that $f$ is globally invertible if
$$
\int_0^{\infty} \inf_{\Vert x \Vert \leq t}  D^{-}_x f \, {\rm d}t = \infty.
$$
Further results along this line have been obtained  in \cite{GutuJaramillo} and \cite{GaGuJa}  in the more general setting of mappings between metric spaces. In a different direction, for global inversion results in terms of Palais-Smale conditions, we refer to \cite{mopre}, \cite{gutuchang} and  \cite{gjps}, and references therein.

\

The notion of pseudo-Jacobian has been recently extended in \cite{JaLaMa} to the case of a continuous map  $f:X\to Y$  between infinite-dimensional Banach spaces, in order to obtain local and global invertibility results, in terms of the metric properties of the corresponding pseudo-Jacobian. Our aim in this paper is to study the global invertibility of locally Lipschitz maps between (possibly infinite-dimensional) Finsler manifolds, stressing the connections with their covering properties and with metric regularity. To this end, we will further extend the notion of {\it pseudo-Jacobian} to this setting. In particular, our results here extend and encompass the previous results of \cite{JaMaSa} and \cite{JaMaSa1} (extending them to an infinite-dimensional setting) and those of \cite{JaLaMa} and \cite{guturmc} (extending them to a Finsler manifold setting). The contents of the paper are as follows. In Section 2, we briefly review the definition and basic properties of pseudo-Jacobians in Banach spaces. This notion is extended  to the setting of Finsler manifolds in Section 3, where several examples are also presented. The connection of metric regularity with metric properties of pseudo-Jacobians is studied in Section 4. Here we introduce the fundamental notion of {\it Finsler regularity index} and we obtain in Theorem \ref{lemmaprincipal} its connection with the  {\it metric rate of surjection}. In Section 5 we introduce the {\it Finsler local-injectivity index} and, combining its properties with the results of the previous section, we obtain in Theorem \ref{inversemapping} a local inversion theorem in this context. Section 6 is devoted to global invertibility. This is achieved by first obtaining conditions under which our map is a covering. The main results here are Theorem \ref{covering1} and Theorem \ref{coveringweighted}. We also obtain in Corollary \ref{hadamardcondition} a version of Hadamard integral condition in this context. Finally, In Section 7 we study the stability of global invertibility under a kind of perturbation with small  Lipschitz constant, as seen in Theorem \ref{lipschitzperturbation}.


\section{\bfseries\sffamily\large Brief review of pseudo-Jacobians in Banach spaces}

In this introductory section, we will briefly recall some basic facts about pseudo-Jacobians associated to nonsmooth mappings between Banach spaces. This notion was studied in \cite{JaLaMa},  and it is the extension to the setting of arbitrary Banach spaces of the pseudo-Jacobian matrices of Jeyakumar and Luc (see \cite{JL}). In what follows, $E$ and $F$ will denote Banach spaces and $U$ a (nonempty) open subset of $E$. As usual, $E^*$ will stand for the topological dual of $E$, and the space of bounded linear operators from $E$ into $F$ will be denoted as $\mathcal{L}(E,F)$. As we will see,  is natural to consider on this space the \emph{weak operator topology} ({\tiny \emph{WOT}} for short),  that is, the topology of the pointwise convergence on $E$ with respect to the weak topology on $F$; this means that a net $(T_i)_{\alpha}$   is {\tiny \emph{WOT}}-convergent to $T$  in $\mathcal{L}(E,F)$ if, and only if, for each $y^*\in F$ and each $v\in E$ the net $(\left\langle y^*,\, (T_{\alpha}-T)(v)\right\rangle)_{\alpha}$ converges to zero.

Finally recall that, if $\varphi :U \to \mathbb R$ is a real-valued function and  $x$ is a point in $U$, then the \emph{upper} and \emph{lower right-hand Dini derivatives} of $\varphi$ at  $x$ with respect to a vector $v\in E$ are defined as
$$
\varphi'_{+}(x;\, v) = \limsup_{t\to 0^+} \frac{\varphi(x+tv)-\varphi(x)}{t}
\quad \text{and
} \quad
\varphi'_{-}(x;\, v) = \liminf_{t\to 0^+} \frac{\varphi(x+tv)-\varphi(x)}{t}.
$$
We refer to \cite{Clarkebook}, \cite{Fabian},  \cite{Ph} or \cite{shapiro} for the definition and basic properties of different kinds of differentiability of maps between Banach spaces and other unexplained notions.

\

Now we are ready to introduce the notion of pseudo-Jacobian in the setting of Banach spaces.

\begin{definition}\label{pseudo-Jacobian}{\rm \bf [Pseudo-Jacobians on Banach spaces}
{\rm Let $E$ and $F$ be Banach spaces,  $U$ be an open subset of $E$ and $f: U \to F$  a continuous map. We say that a nonempty subset  $Jf(x)\subset \mathcal{L}(E, F)$ is {\it a pseudo-Jacobian} of $f$ at a point $x\in U$ if
\begin{equation}
(y^* \circ f)'_{+}(x;\, v) \leq \sup \{ \langle y^*,\, T(v) \rangle \, : \, T\in Jf(x) \} \quad \text{whenever} \,\, y^* \in F^* \,\, \text{and} \,\,  v\in E.
\end{equation}
A set-valued mapping $Jf:U\to 2^{\mathcal{L}(E,F)}$ is said to be  {\it a pseudo-Jacobian mapping} for $f$ on $U$ if for every $x\in U$, the set $Jf(x)$ is a pseudo-Jacobian of $f$ at $x$.}
\end{definition}

Let us summarize some basic facts about pseudo-Jacobians. It is clear that, if $Jf(x)$ is a pseudo-Jacobian of $f$ at the point $x$, then any subset of $\mathcal{L}(E,F)$ containing $Jf(x)$ is also a pseudo-Jacobian of $f$ at $x$.  Moreover, for any set $\mathcal{F}\subset \mathcal{L}(E, F)$ we have
$$
\sup \left\{\left\langle y^*,\, T(v)\right\rangle \, : \, T\in \mathcal{F} \right\}= \sup \left\{\left\langle y^*,\, T(v)\right\rangle \, :\, T\in \overline{\co}^{WOT} (\mathcal{F})\right\},
$$
where $\overline{\co}^{WOT} (\mathcal{F})$ denotes the {\tiny \emph{WOT}}-closed convex hull of $\mathcal{F}$. Thus, a subset $Jf(x)\subset \mathcal{L}(E, F)$ is a pseudo-Jacobian of $f$ at $x$ if, and only if, so is its {\tiny \emph{WOT}}-closed convex hull.

\

The same argument as in the finite-dimensional case (see \cite[Theorem 2.1.1]{JL}) yields that if $Jf(x)$ and $Jg(x)$ are pseudo-Jacobians of functions $f,\, g:U\subset X\to Y$ at a point $x\in U$ and $\alpha\in \mathbb R$, then the set $\alpha Jf(x) + Jg(x) = \{\alpha T+S:\, T\in Jf(x),\, S\in Jg(x)\}$ is a pseudo-Jacobian of $\alpha f + g$ at this point.\\

There are many examples pseudo-Jacobians for different kinds of functions. For instance, if the function $f$ is G\^ateaux differentiable at $x$ then the singleton $\{f'(x)\}$ is a pseudo-Jacobian of $f$ at $x$ (see \cite[Example 2.2]{JaLaMa}). More generally, according to the definition given by Ioffe in \cite{Ioffe}, the G\^ateaux prederivative of $f$ at the point $x$ is also a pseudo-Jacobian of $f$ at $x$ (see \cite[Example 2.3]{JaLaMa}).\\

On the other hand, if the function $f:E \to \mathbb R$ is locally Lipschitz at a point $x$, the Clarke subdifferential $\partial f(x)$ is a pseudo-Jacobian of $f$ at $x$. More generally, in the vector-valued case,  the so-called Clarke-like generalized Jacobians (as considered in \cite{Thibault}) are pseudo-Jacobians of $f$ at $x$ (see \cite[Examples 2.4 and 2.5]{JaLaMa}). Regarding to Clarke-like generalized Jacobians, for a locally Lipschitz map $f:E \to F$, where $F$ is reflexive, the P\'ales-Zeidan generalized Jacobian $\partial_{PZ}f(x)$ as defined in \cite{PZ-07}, is a pseudo-Jacobian of $f$ at every point $x\in E$ (see \cite[Examples 2.6]{JaLaMa}). Note that, in the case that $E$ and $F$ are finite-dimensional, the P\'ales-Zeidan generalized Jacobian coincides with the Clarke generalized Jacobian. Let us briefly recall the definition of $\partial_{PZ}f(x)$. Note that, in our case, since the space $F$ is reflexive it has the Radon-Nikod\'ym property, and furthermore the topology $\beta (E, V)=\beta(E, F^*)$ considered in \cite{PZ-07} coincides with the {\tiny \emph{WOT}}-topology on ${\mathcal L}(E, F)$. Now given a finite-dimensional linear subspace $L\subset E$, we say that $f$ is {\it $L$-G\^ateaux-differentiable} at  $z\in E$ if there exists a continuous linear map $D_L(z):L\to F$ such that
$$
\lim_{t \to 0} \frac{f(z+tv) - f(z)}{t}= D_L f(z)(v), \quad \text{for every} \quad v\in L.
$$
Denote by $\Omega_L(f)$ the set of all points $z\in E$ such that $f$ is $L$-G\^ateaux-differentiable at $z$, and let the {\it generalized L-Jacobian} of $f$ at $x$ be the subset of $\mathcal{L}(L,\, F)$ defined as
$$
\partial_L f(x) := \bigcap_{\delta>0} \overline{co}^{\tiny WOT} \{ D_L f(z) \, : \,  z \in \Omega_L(f), \, |z-x|_E <\delta \},
$$
where $|\cdot|_E$ denotes the norm of $E$. Then the \emph{P\'ales-Zeidan generalized Jacobian of $f$ at the point $x$} is defined as
$$\partial_{PZ} f(x) = \left\{T\in \mathcal{L}(E,\, F):\,\, T_{|_ L}\in \partial_L f(x),\,\, \text{for each finite dimensional subspace}\, L\subset X\right\}.$$

\

The following property of pseudo-Jacobians will be useful in the sequel.

\begin{definition}{\rm \bf [Locally bounded pseudo-Jacobian]}
{\rm Let $E, F$ be Banach spaces,  $U$  an open subset of $E$ and  $f:U\to F$ a continuous map. A pseudo-Jacobian mapping  $Jf: U \to 2^{{\mathcal L}(E, F)}$ for $f$ is said to be  \emph{locally bounded}
at a point $x\in U$, if there exists $R>0$ such that the ball $B(x, R) \subset U$ and the set of operators
$$Jf\left(B(x, R) \right) = \left\{T \, : \, T\in Jf(z),\, z\in B(x, R) \right\}$$
is bounded in the space $\mathcal{L}(E,F)$. We say that $Jf$ is \emph{locally bounded}
\emph{on}  $U$ if this holds for every $x\in U$.}
\end{definition}

As shown in \cite[Corollary 2.8]{JaLaMa}, a continuous map $f:U\to F$ admits a locally bounded pseudo-Jacobian mapping if, and only if, $f$ is locally Lipschitz.


\section{\bfseries\sffamily\large Pseudo-Jacobians on Finsler manifolds}

Along the paper we will consider $C^1$-smooth manifolds modeled on (possibly infinite-dimensional) Banach spaces. We will follow the terminology of Palais in \cite{palais}. For the definition of pseudo-Jacobians in this context, we will restrict ourselves to the case of locally Lipschitz mappings, which can be defined by means of composition with charts.

\subsection*{Locally Lipschitz continuous maps: 1st definition \cite[Definitions 1.1 and 1.3]{palais}}
Let $X$ and $Y$  be two $C^1$ manifolds modeled on Banach spaces $(E, |\cdot|_E)$ and  $(F, |\cdot|_F)$, respectively, and  let $f:X\rightarrow Y$ be a map. We say that $f$ is {\it locally Lipschitz}  at $x\in X$ if there are charts $(W,\varphi)$ at $x$ and $(V,\psi)$ at $f(x)$, such that $f(W)\subset V$ and the map
$$\mathbf{f} = \psi \circ f \circ \varphi^{-1}:\varphi(W)\rightarrow \psi(V)$$
is Lipschitz continuous on $\varphi(W)$. Namely for all $\mathbf{u},\mathbf{u}'\in\varphi(W)$ and some $\kappa>0$:
\begin{equation}\label{lipcondition}
|\mathbf{f}(\mathbf{u})-\mathbf{f}(\mathbf{u}')|_F\leq \kappa |\mathbf{u}-\mathbf{u}'|_E.
\end{equation}
Obviously, the above definition does not depend on the choice of charts.

Our definition of pseudo-Jacobian for functions between manifolds will be also given by composition with charts. In order to have a good behavior with respect to this composition, in the definition we require the corresponding pseudo-Jacobian mapping between Banach spaces to be locally bounded. If $X$ is a $C^1$ manifold modeled on Banach space an each $x\in X$, the tangent space of $X$ at the point $x$ will be denoted by $T_xX$.

\begin{definition}\label{PJ}{\rm \bf [Pseudo-Jacobians mappings on manifolds]} {\rm
Let $X$ and $Y$  be two  $C^1$ manifolds modeled on Banach spaces $E$ and  $F$, respectively, and let  $f:X\rightarrow Y$  be a locally Lipschitz  map. Suppose that, for each $x\in X$ we have a subset $Jf(x)$ of linear operators from $T_xX$ to $T_{f(x)}Y$. We say that $Jf$ is {\it a pseudo-Jacobian mapping} for $f$   if for each $x\in X$ there exists a chart $(W,\varphi)$ at $x$ and a chart $(V,\psi)$ at $f(x)$ such that $f(W)\subset V$ and:
\begin{itemize}
\item[(PJ$_1$)] The function
\begin{equation}\label{boldf}
\mathbf{f} = \psi \circ f \circ \varphi^{-1}:\varphi(W)\rightarrow \psi(V)
\end{equation} has a pseudo-Jacobian $J\mathbf{f}(\mathbf{u})\subset \mathcal{L}(E,F)$ at every point $\mathbf{u}\in\varphi(W)$.
\item[(PJ$_2$)]  The pseudo-Jacobian mapping $J\mathbf{f}: \varphi(W)\rightarrow 2^{\mathcal{L}(E,F)}$ is  locally bounded on $\varphi(W)$.
\item[(PJ$_3$)]  If $\mathbf{y}=\psi(f(x))\in\psi(V)$, then
$$Jf(x)=d\psi^{-1}(\mathbf{y}) J\mathbf f(\mathbf{x})[d\varphi^{-1}(\mathbf{x})]^{-1}$$
\end{itemize}
This means that every $T\in Jf(x)\subset L(T_xX,T_{f(x)}Y)$ is of the form $d\psi^{-1}(\mathbf{y}) \hspace{0.01in}\mathbf{T} \hspace{0.01in}[d\varphi^{-1}(\mathbf{x})]^{-1}$ where  $\mathbf{T}\in  J\mathbf f(\mathbf{x})\subset \mathcal{L}(E,F)$
and vice-versa:

\begin{center}
\begin{tikzpicture}
  \node (A) { \hspace{0.2in}$E$ \hspace{0.2in}};
  \node (B) [below=of A] {$T_xX$};
  \node (C) [right=of A] { \hspace{0.1in}$F$ \hspace{0.1in} };
  \node (D) [right=of B] {$T_{f(x)}Y$};
  \draw[stealth-] (A)-- node[left] {\small \begin{tabular}{c} Isomorphism \\ $[d\varphi^{-1}(\mathbf{x})]^{-1}$ \end{tabular}} (B);
  \draw[-stealth] (B)-- node [below] {\small $T$} (D);
  \draw[-stealth] (A)-- node [above] {\small $\mathbf{T}$} (C);
  \draw[-stealth] (C)-- node [right] {\small \begin{tabular}{c} Isomorphism \\ $d\psi^{-1}(\mathbf{y})$\end{tabular}} (D);
\end{tikzpicture}
\end{center}}
\end{definition}

The stability properties of pseudo-Jacobians under composition with smooth functions, as obtained in \cite{JaLaMa}, yield the following result.

\begin{proposition}
The definition of pseudo-Jacobian on Banach manifolds does not depend on charts {\rm in the following sense: let  $(W_1,\varphi_1)$  be a chart at $x$ and $(V_1,\psi_1)$ be a chart at $f(x)$  such that $f(W_1)\subset V_1$ and  {\rm (PJ$_1$), (PJ$_2$)} and {\rm (PJ$_3$)} holds for
$\mathbf{f}_1 = \psi_1 \circ f\circ \varphi_1^{-1}:\varphi_1(W_1)\rightarrow \psi_1(V_1)$,  $\mathbf{x}_1=\varphi_1(x)$ and  $\mathbf{y}_1=\psi_1(f(x))$. If  $(W_2,\varphi_2)$  and $(V_2,\psi_2)$  are another charts around $x$ and $f(x)$ respectively,  such that $f(W_2)\subset V_2$, then there is a locally bounded pseudo-Jacobian mapping
$J\mathbf{f_2}: \varphi_2(W_2)\rightarrow 2^{\mathcal{L}(E,F)}$ such that:
$$d\psi_1^{-1}(\mathbf{y}_1) J\mathbf f_1(\mathbf{x}_1)[d\varphi_1^{-1}(\mathbf{x}_1)]^{-1}=d\psi_2^{-1}(\mathbf{y}_2) J\mathbf f_2(\mathbf{x}_2)[d\varphi_2^{-1}(\mathbf{x}_2)]^{-1},$$
where $\mathbf{f}_2 = \psi_2 \circ f\circ \varphi_2^{-1}:\varphi_2(W_2)\rightarrow \psi_2(V_2)$,  $\mathbf{x}_2=\varphi_2(x)$ and $\mathbf{y}_2=\psi_2(f(x))$. }
\end{proposition}

\begin{proof}
Indeed, consider the following $C^1$-diffeomorphisms:
\begin{eqnarray*}
\Psi &=& \psi_2\circ\psi_1^{-1}:\psi_1(V_1\cap V_2)\rightarrow F\\
\Phi &=& \varphi_1\circ\varphi_2^{-1}: \varphi_2(W_1 \cap W_2)\rightarrow E
\end{eqnarray*}
\noindent By \cite[Proposition 2.12]{JaLaMa} $J\mathbf{g}(\mathbf{x}):=J\mathbf{f}_1(\Phi(\mathbf{x}))\circ d\Phi(\mathbf{x})$ is a pseudo-Jacobian for $\mathbf{g} :=\mathbf{f}_1\circ \Phi:  \varphi_2(W_2)\rightarrow \psi_1(V_1)$ at every $\mathbf{x}\in \varphi_2(W_2)$. Furthermore the set-valued function $\mathbf{x}\mapsto J\mathbf{g}(\mathbf{x})$ is a locally bounded pseudo-Jacobian mapping since $\Phi$ is a diffeomorphism. Now, by \cite[Theorem 2.15]{JaLaMa}, the set
$d\Psi(\mathbf{g}(\mathbf{x}))\circ J\mathbf{g}(\mathbf{x})$
is a pseudo-Jacobian of $\Psi\circ\mathbf{g}$ at $\mathbf{x}$, and therefore
$$J({\Psi\circ\mathbf{f}_1\circ \Phi})(\mathbf{x})= d\Psi(\mathbf{g}(\mathbf{x}))\circ J\mathbf{f}_1(\Phi(\mathbf{x}))\circ d\Phi(\mathbf{x})$$
is a pseudo-Jacobian of $\Psi\circ\mathbf{f}_1\circ \Phi$ at $\mathbf{x}\in\varphi_2(W_2)$. Simple calculations show that $\Psi\circ\mathbf{f}_1\circ \Phi = \mathbf{f_2}$, $J\mathbf{f}_1(\Phi(\mathbf{x}_2)) =  J\mathbf{f}_1 (\mathbf{x}_1)$,
$d\Psi(\mathbf{g}(\mathbf{x}_2)) = [d \psi_2^{-1} (\mathbf{y}_2)]^{-1} \circ d\psi_1^{-1}(\mathbf{y}_1)$ and $d\Phi(\mathbf{x}_2) =  [d \varphi_1^{-1}(\mathbf{x_1})]^{-1}\circ d\varphi_2^{-1}(\mathbf{x}_2)$. Therefore:
\begin{eqnarray*}
J \mathbf{f_2}(\mathbf{x_2}) & = &  d\Psi(\mathbf{g}(\mathbf{x}_2))\circ J\mathbf{f}_1(\Phi(\mathbf{x}_2))\circ d\Phi(\mathbf{x_2})\\
& = &  [d \psi_2^{-1} (\mathbf{y}_2)]^{-1} \circ d\psi_1^{-1}(\mathbf{y}_1)\circ  J\mathbf{f}_1 (\mathbf{x}_1)\circ  [d \varphi_1^{-1}(\mathbf{x_1})]^{-1}\circ d\varphi_2^{-1}(\mathbf{x}_2).
\end{eqnarray*} The set-valued function  $\mathbf{x}\mapsto J{\mathbf{f}_2}(\mathbf{x}) = J({\Psi\circ\mathbf{f}_1\circ \Phi})(\mathbf{x})$ is a locally bounded pseudo-Jacobian mapping since $\Psi$ and $\Phi$ are $C^1$-diffeomorphisms.
\end{proof}

\subsection*{Examples} We give here several  natural examples of pseudo-Jacobians for a locally Lipschitz map in the setting of Banach manifolds.
\begin{enumerate}
\item {\bf [G\^{ateaux} derivative]} Let $X$ and $Y$ be two $C^1$ Banach manifolds, modeled on Banach spaces $E$ and $F$, respectively. We say that a locally Lipshitz map $f:X \to Y$ is {\it G\^{ateaux} differentiable} at a point $x\in X$ if there exist charts $(W,\varphi)$ at $x$ and $(V,\psi)$ at $f(x)$ such that $f(W)\subset V$ and the map $\mathbf{f} := \psi \circ f \circ \varphi^{-1}:\varphi(W)\rightarrow \psi(V)$ is G\^{ateaux} differentiable at $\mathbf{x}=\varphi(x)$. In this case the {\it G\^{ateaux} derivative} of $f$ at $x$, denoted by $df(x)$ is defined as the linear map $T: T_xX \to T_{f(x)}Y$ given by
    $$
    T = d(\psi^{-1})(\psi (f(x)) \circ d\mathbf{f}(\mathbf{x}) \circ d \varphi (x),
    $$
    where $d\mathbf{f}(\mathbf{x})$ denotes the usual G\^{ateaux} derivative of $\mathbf{f}$ at $\mathbf{x}$ in the setting of Banach spaces. In order to see that the above definition does not depend on charts note that,  for locally Lipschitz maps between Banach spaces, G\^{ateaux} differentiability is equivalent to Hadamard differentiability (see e.g. Proposition 3.5 in \cite{shapiro}) and note also that the usual chain rule holds for Hadamard differentiability (see e.g. Proposition 3.6 in \cite{shapiro}). In this way, and using \cite[Example 2.2]{JaLaMa},  we  obtain that if $f$ is G\^{ateaux} differentiable at every point of $X$ and we define $Jf(x):= \{df(x)\}$ for each $x\in X$,  then $Jf$ is a pseudo-Jacobian mapping for $f$.

\item {\bf [Clarke generalized Jacobian]} In the case that $X$ and $Y$ are finite-dimensional $C^1$ manifolds, and $f:X \to Y$ is locally Lipschitz, the {\it Clarke generalized Jacobian}  of $f$ at a point $x\in X$ has been considered in \cite{JaMaSa} and it is defined as:
    $$
    \partial f(x):= d(\psi^{-1})(\psi (f(x)) \circ \partial\mathbf{f}(\mathbf{x}) \circ d \varphi (x),
    $$
    where $(W,\varphi)$ is a chart of $X$ at $x$ and $(V,\psi)$ is a chart of $Y$ at $f(x)$ with $f(W)\subset V$, we denote as before $\mathbf{f}:= \psi \circ f \circ \varphi^{-1}$ and $\partial\mathbf{f}(\mathbf{x})$ denotes the usual Clarke generalized Jacobian of $\mathbf{f}$ at $\mathbf{x}=\varphi(x)$. It is proved in Proposition 2.3 of \cite{JaMaSa} that the definition does not depend on charts. Therefore, if we define $Jf(x):= \partial f(x)$ for each $x\in X$, using \cite[Example 2.5]{JaLaMa} we obtain that $Jf$ is a pseudo-Jacobian mapping for $f$.

\item {\bf [P\'ales-Zeidan generalized Jacobian]} Let $X$ and $Y$ be two $C^1$ Banach manifolds, modeled on Banach spaces $E$ and $F$ respectively, where $F$ is reflexive, and let $f:X \to Y$ a locally Lipschitz map. The {\it P\'ales-Zeidan generalized Jacobian} of $f$ at a point $x\in X$  is defined as:
    $$
    \partial_{PZ} f(x):= d(\psi^{-1})(\psi (f(x)) \circ \partial_{PZ}\mathbf{f}(\mathbf{x}) \circ d \varphi (x),
    $$
    where $(W,\varphi)$ is a chart of $X$ at $x$ and $(V,\psi)$ is a chart of $Y$ at $f(x)$ with $f(W)\subset V$, we denote as before $\mathbf{f}:= \psi \circ f \circ \varphi^{-1}$ and $\partial_{PZ}\mathbf{f}(\mathbf{x})$ denotes the usual P\'ales-Zeidan generalized Jacobian of $\mathbf{f}$ at $\mathbf{x}=\varphi(x)$ in the Banach space setting. We are going to see in Proposition \ref{PZ-equivalence} that the above definition does not depend on charts. Then using again  \cite[Example 2.5]{JaLaMa}, if we define $J_{PZ}f(x):= \partial_{PZ} f(x)$ for each $x\in X$ we obtain that  $Jf$ is a pseudo-Jacobian mapping for $f$.

\end{enumerate}

\

For completeness, we include the following elementary Lemma.

\begin{lemma}\label{simple}
Let $U$ and $V$ be open subsets of the Banach spaces $E$ and $F$, respectively. Suppose that $\varphi : U\to E$ and $\psi : V\to F$ are $C^1$-smooth and $f:W \to V$ is a Lipschitz map defined on a open set $W$ containing $\varphi (U)$. Fix a point $a \in U$ and a finite-dimensional linear subspace $L\subset E$, and denote $K:=d\varphi(a)(L)$. If $f$ is G\^{a}teaux-K-differentiable at $\varphi (a)$, then $\psi \circ f \circ \varphi$ is G\^{a}teaux-L-differentiable at $a$, and the chain rule holds:
$$
D_L (\psi \circ f \circ \varphi) (a) = d \psi (f(\varphi (a)) \circ D_K f(\varphi (a)) \circ d\varphi (a).
$$
\end{lemma}
\begin{proof}
Consider first the composition $f \circ \varphi : U \to F$. Suppose that $f$ is $M$-Lipschitz.  Choose a vector $v\in L$ and denote $w=d\varphi (a)$. Then, for $t\neq 0$:
$$
\left \Vert \frac{ f \circ \varphi (a+tv) - f\circ \varphi (a)}{t} - D_K f (\varphi (a))\circ d\varphi (v) \right \Vert
$$
$$
\leq \left \Vert \frac { f(\varphi (a+tv)) - f(\varphi (a) + t w)}{t} \right \Vert +
\left \Vert \frac{f(\varphi (a) + t w) - D_K f (\varphi (a)) (w)}{t}\right \Vert
$$
$$
\leq M \, \left \Vert \frac { \varphi (a+tv) - \varphi (a) + t d\varphi (a)(v))}{t} \right \Vert + \left \Vert \frac{f(\varphi (a) + t w) - D_K f (\varphi (a)) (w)}{t}\right \Vert,
$$
and this tends to $0$ as $t \to 0$. This shows that  $f \circ \varphi$ is G\^{a}teaux-L-differentiable at $a$.

Consider now the map $g: L\cap (-a+U) \to E$ defined by $g(x):= f \circ \varphi (a+x)$. Since $g$ is G\^{a}teaux differentiable at $0$, and furthermore $g$ is locally Lipschitz,  we have that $g$ is Hadamard differentiable at $0$. In fact, since $g$ is defined on a finite-dimensional space, $g$ is Fr\'echet differentiable at $0$ (see e.g. Propositions 3.5 and 3.6 of \cite{shapiro}). As a consequence, $\psi \circ g$ is Fr\'echet differentiable at $0$ and the usual chain rule holds. This gives the desired result.
\end{proof}

Now we are ready to show that the P\'ales-Zeidan generalized Jacobian on manifolds is well-defined.

\begin{proposition}\label{PZ-equivalence}
{\rm The definition of P\'ales-Zeidan generalized Jacobian on Banach manifolds does not depend on charts.}
\end{proposition}
\begin{proof}
We keep the preceding  notation. Let $X$ and $Y$ be manifolds modeled on Banach spaces $E$ and $F$, respectively. Consider a locally Lipschitz map $f:X \to Y$ and let  $x\in X$.
Let $(W_1,\varphi_1)$  be a chart at $x$ and $(V_1,\psi_1)$ be a chart at $f(x)$  with $f(W_1)\subset V_1$. Denote  $\mathbf{f}_1 = \psi_1 \circ f\circ \varphi_1^{-1}:\varphi_1(W_1)\rightarrow \psi_1(V_1)$,  $\mathbf{x}_1=\varphi_1(x)$ and  $\mathbf{y}_1=\psi_1(f(x))$.

Now let $(W_2,\varphi_2)$  and $(V_2,\psi_2)$  also be charts around $x$ and $f(x)$ respectively,  such that $f(W_2)\subset V_2$, and denote $\mathbf{f}_2 = \psi_2 \circ f\circ \varphi_2^{-1}:\varphi_2(W_2)\rightarrow \psi_2(V_2)$,  $\mathbf{x}_2=\varphi_2(x)$ and $\mathbf{y}_2=\psi_2(f(x))$.

As before, we will also consider the following $C^1$-diffeomorphisms:
\begin{eqnarray*}
\Psi &=& \psi_2\circ\psi_1^{-1}:\psi_1(V_1\cap V_2)\rightarrow F\\
\Phi &=& \varphi_1\circ\varphi_2^{-1}: \varphi_2(W_1 \cap W_2)\rightarrow E
\end{eqnarray*}
In this way we have that
$$
\mathbf{f}_2 = \Psi \circ \mathbf{f}_1 \circ \Phi
$$

Now, if we denote
$$
\partial_{1} f(x):= d(\psi_1^{-1})(\psi_1 (f(x)) \circ \partial_{PZ}\mathbf{f}_1(\mathbf{x}_1) \circ d \varphi_1 (x), \text{\, \, \, and}
$$
$$
\partial_{2} f(x):= d(\psi_2^{-1})(\psi_2 (f(x)) \circ \partial_{PZ}\mathbf{f}_2(\mathbf{x}_2) \circ d \varphi_2 (x),
$$
we have to prove that $\partial_{1} f(x)= \partial_{2} f(x)$. By symmetry, it will be sufficient to prove that
$$
\partial_{1} f(x) \subset \partial_{2} f(x).
$$
This is equivalent to prove that
$$
B\circ \partial_{PZ}\mathbf{f}_1(\mathbf{x}_1) \circ A \subset \partial_{PZ}\mathbf{f}_2(\mathbf{x}_2),
$$
where $A:= d \Phi (\mathbf{x}_2)$ and $B:= d \Psi (\mathbf{y}_1)$. By the definition of P\'ales-Zeidan generalized gradient, this will be a direct consequence of the following claim:

\

\noindent {\bf Claim:} For each finite-dimensional linear subspace $L\subset E$, if we denote $K= A(L)$, we have that
$$
B\circ \partial_{K}\mathbf{f}_1(\mathbf{x}_1) \circ A \subset \partial_{L}\mathbf{f}_2(\mathbf{x}_2).
$$

In order to prove the Claim, we are going to use that characterization of generalized $L$-Jacobians given in  Lemma 3.2 of \cite{PZ-07}. According to it,  the generalized $K$-Jacobian of $\mathbf{f}_1$ at $\mathbf{x}_1$ is given by
$$
\partial_{K}\mathbf{f}_1(\mathbf{x}_1) = \overline{co}^{\tiny WOT}(\Delta_{K}\mathbf{f}_1(\mathbf{x}_1)),
$$
where
$$
\Delta_{K}\mathbf{f}_1(\mathbf{x}_1) = \{ T\in \text{{\tiny {\emph WOT}}-cluster} D_K \mathbf{f}_1(\mathbf{z}_n) \, : \, (\mathbf{z}_n)\in \Omega_K(\mathbf{f}_1), \, (\mathbf{z}_n)\to \mathbf{x}_1 \}.
$$

This means that an operator $T\in \mathcal L (K, F)$ belongs to $\Delta_{K}\mathbf{f}_1(\mathbf{x}_1)$ if, and only if, $T$ is the limit in the {\tiny {\emph WOT}}-topology of a {\it subnet} of the form $\{D_K \mathbf{f}_1(\mathbf{z}_{n_\alpha})\}$ where $(\mathbf{z}_n)$ is a sequence in $\Omega_K(\mathbf{f}_1)$ converging to $\mathbf{x}_1$.

Suppose this is the case. Then, if we denote $\mathbf{w}_{n}:= \Phi^{-1}(\mathbf{z}_{n})$, we have that $(\mathbf{w}_{n})$ converges to $\mathbf{x}_2$ and the net $A_{{n_\alpha}}:= d \Phi (\mathbf{w}_{n_\alpha})$ converges to $A$ in norm. In the same way, if we denote $\mathbf{u}_{n}:= \mathbf{f}_1(\mathbf{z}_{n})$,
we have that $(\mathbf{u}_{n})$ converges to $\mathbf{y}_1$ and the net $B_{{n_\alpha}}:= d \Phi (\mathbf{u}_{n_\alpha})$ converges to $B$ in norm. Furthermore, using  Lemma \ref{simple}, for each index $\alpha$ we have that $\mathbf{f}_2$ is G\^{a}teaux-L-differentiable at $\mathbf{w}_{n_\alpha}$ and
$$
B_{n_\alpha} \circ D_K \mathbf{f}_1(\mathbf{z}_{n_\alpha}) \circ A_{n_\alpha} = D_L \mathbf{f}_1(\mathbf{w}_{n_\alpha}).
$$
Since $\mathbf{f}_1$ is locally Lipschitz, we can also assume that the net $D_K \mathbf{f}_1(\mathbf{z}_{n_\alpha})$ is norm-bounded in $\mathcal L (K, F)$. On the other hand, from Proposition 2.2 in \cite{PZ-07} we know that the mapping $R\mapsto R\circ A$ is a linear isomorphism for the respective {\tiny {\emph WOT}}-topologies. In this way we obtain that the net $\{D_K \mathbf{f}_1(\mathbf{z}_{n_\alpha}) \circ A_{n_\alpha}\}$ converges to $T\circ A$ in the {\tiny {\emph WOT}}-topology. By the same reasoning, we see that the net
$$
\{B_{n_\alpha} \circ D_K \mathbf{f}_1(\mathbf{z}_{n_\alpha}) \circ A_{n_\alpha}\}
$$
converges to $B\circ T \circ A$ in the {\tiny {\emph WOT}}-topology. This gives that  $\Delta_{K}\mathbf{f}_1(\mathbf{x}_1)\subset \partial_{L}\mathbf{f}_2(\mathbf{x}_2)$. By the very definition, the latter is a {\tiny {\emph WOT}}-closed and convex set, so this establishes the Claim.

\end{proof}

\subsection*{Finsler metrics \cite[p. 116]{palais}} Let $X$ be a $C^1$ manifold modeled in a  Banach space $(E,|\cdot|_E)$. As usual, $TX=\{(x,v):x\in X\mbox{ and } v\in T_xX\}$ will denote the tangent bundle of $X$.  If $(W,\varphi)$ is a chart of $X$, then there is a local trivialization of the natural projection $\pi:TX\rightarrow X$ over $W$, namely a  bijection $TW=\pi^{-1}(W)\rightarrow W\times E$ which commutes  with the projection on $W$. For every $\mathbf{v}\in E$ and $x\in W$, $d\varphi^{-1}(\varphi(x))\mathbf{v}\in T_xX$ is the tangent vector at $x$ represented by $\mathbf{v}$ in the chart $(W,\varphi)$.  A {\it Finsler structure} on $TX$ is a continuous map $\|\cdot\|:TX\rightarrow[0,\infty)$ such that:

\begin{enumerate}
\item For every $x\in X$, the map $\|\cdot\|_x:= \|\cdot\|_{|_{T_xX}}:T_xX\rightarrow[0,\infty)$ is a norm on $T_xX$ such that, for  every chart $(W,\varphi)$ at $x$,  the map $\|d\varphi^{-1}(\varphi(u))(\cdot)\|_x$ is a norm on $E$ equivalent to $|\cdot|_E$.
\item Given $x_0\in X$,  a chart $(W,\varphi)$ of $X$ at $x_0$ and $\varepsilon> 0$ there exists an open neighborhood $U_{x_0}\subset W$  such that for every $x\in U_{x_0}$ and every $\mathbf{v}\in E$:
$$\frac{1}{(1+\varepsilon)} \|d\varphi^{-1}(\varphi(x_0))(\mathbf{v})\|_{x_0}\leq \|d\varphi^{-1}(\varphi(x))(\mathbf{v})\|_x\leq (1+\varepsilon)\|d\varphi^{-1}(\varphi(x_0))(\mathbf{v})\|_{x_0}.$$
\end{enumerate}
A {\it Finsler manifold} is a $C^1$-smooth Banach manifold endowed with a Finsler structure on its tangent bundle.

\

Let $X$ be a $C^1$  Finsler manifold. Recall that the {\it length} of a  $C^1$-smooth path $\sigma:[a,b]\rightarrow X$ is defined as
$$\ell(\sigma)=\int_a^b\|\dot \sigma(t)\|_{\sigma(t)} dt.$$
If $X$ is connected, then it is connected by $C^1$-smooth paths and we can define the associated {\it Finsler metric}:
$$d_X(u,u')=\inf\{\ell(\sigma):\sigma \mbox{ is a  $C^1$-smooth path connecting $u$ to $u'$} \}.$$
The Finsler metric is consistent with the topology given in $X$ and the manifold is said to be {\it complete} if it is a complete metric space with respect to the distance $d_X$. {\it From now on we will assume that all Finsler manifolds are connected}.

\

The next lemma shows that the Finsler distance can be locally approximated by the norm-distance associated to a given chart on the model space.

\begin{lemma}\label{lemaequivalencia1}
Let $X$  be a  $C^1$ Finsler manifold modeled on a Banach space $E$. Given $x\in X$,  a chart $(W,\varphi)$ of $X$ at $x$  and $\varepsilon>0$ there is  an open neighborhood $U_x\subset W$ such that:
\begin{equation}\label{metricasequivalentes}
\frac{1}{(1+\varepsilon)}\|\mathbf{u}-\mathbf{u}'\|_{x,\varphi} \leq d_X(u,u')\leq (1+\varepsilon) \|\mathbf{u}-\mathbf{u}'\|_{x,\varphi}
\end{equation}
where $\mathbf{u}=\varphi(u)$,
$\mathbf{u}'=\varphi(u')$ for $u,u'\in U_x$ and  $\|\cdot\|_{x,\varphi} := \|d\varphi^{-1}(\varphi(x))(\cdot)\|_x$.
\end{lemma}

\begin{proof}
Let $x\in X$ and  $\varepsilon>0$, and a chart $(W,\varphi)$ of $X$ at $x$. Then there is an open and connected neighborhood $U_x$  such that for every $u\in U_x$ and every $\mathbf{v}\in E$:
$$\frac{1}{(1+\varepsilon)}\|\mathbf{v}\|_{x,\varphi}\leq \|\mathbf{v}\|_{u,\varphi}\leq (1+\varepsilon)\|\mathbf{v}\|_{x,\varphi}.$$
Without loss of generality, we can also assume that $\varphi(U_x)$ is a convex open set in $E$. Let $u$ and $u'$ arbitrary points in $U_x$, and let $\sigma:[0,1]\rightarrow U_x$  be a $C^1$-smoth path joining $u$ and $u'$. If $\mathbf{s}=\varphi\circ\sigma$, for each $t\in(0,1)$ we have that
$\dot\sigma(t)
	=  d\varphi^{-1}(\varphi(\sigma(t)))\mathbf{\dot s}(t)$. Then
$$\ell(\sigma)=\int_0^1\|\dot \sigma(t)\|_{\sigma(t)} dt
	= \int_0^1 \|\mathbf{\dot s}(t)\|_{{\sigma(t)},\varphi}dt
	\geq\frac{1}{(1+\varepsilon)}\int_0^1 \|\mathbf{\dot s}(t)\|_{x,\varphi} dt$$
$$\geq\frac{1}{(1+\varepsilon)} \Bigg\|\int_0^1\mathbf{\dot s}(t)dt\Bigg\|_{x,\varphi}
    \geq\frac{1}{(1+\varepsilon)}\|\mathbf{s}(1)-\mathbf{s}(0)\|_{x,\varphi}$$
Therefore $d_X(u,u')\geq\frac{1}{(1+\varepsilon)}\|\mathbf{u}-\mathbf{u'}\|_{x,\varphi}.$
Now, consider the path $\mathbf{s}(t)=t\varphi(u)+(1-t)\varphi(u')$ in $\varphi(U_x)$ and $\sigma(t)=\varphi^{-1}(\mathbf{s}(t))$. Then:
$$\ell(\sigma)
	= \int_0^1\|\dot \sigma(t)\|_{\sigma(t)} dt
	= \int_0^1 \|\mathbf{\dot s}(t)\|_{\sigma(t),\varphi}dt
	\leq {(1+\varepsilon)}\int_0^1 \|\mathbf{\dot s}(t)\|_{x,\varphi} dt$$
$$    = {(1+\varepsilon)}\int_0^1 \|\varphi(u)-\varphi(u')\|_{x,\varphi} dt.$$
Then $d_X(u,u')\leq\ell(\sigma)\leq {(1+\varepsilon)} \|\mathbf{u}-\mathbf{u'}\|_{x,\varphi}.$
\end{proof}

As a consequence, we see that the Finsler distance is locally comparable with the distance in the model space given by the original norm.

\begin{lemma}\label{lemaequivalencia2}
Let $X$  be a  $C^1$ Finsler manifold modeled in a Banach space $(E,|\cdot|_E)$. Given $x\in X$,  a chart $(W,\varphi)$ of $X$ at $x$  and $\varepsilon>0$ there is  an open neighborhood $U_x\subset W$  and $m>0$ such that such that for every $u,u'\in U_x$:
\begin{equation}\label{normasequivalentes}
\frac{1}{m(1+\varepsilon)}|\varphi(u)-\varphi(u')|_E \leq d_X(u,u')\leq m(1+\varepsilon) |\varphi(u)-\varphi(u')|_E.
\end{equation}
\end{lemma}

\begin{proof}
By Lemma \ref{lemaequivalencia1}, for every $x\in X$,  a chart $(W,\varphi)$ of $X$ at $x$, and $\varepsilon>0$ there is an open neighborhood $U_x\subset W$ such that \eqref{metricasequivalentes} holds or every $u,u'\in U_x$. Since  $\|d\varphi^{-1}(\varphi(x))(\cdot)\|_x$ is a norm on $E$ equivalent to $|\cdot|_E$, there exists $m>0$ such that $\frac{1}{m}|\mathbf{v}|_E\leq \|d\varphi^{-1}(\varphi(x))\mathbf{v}\|_x\leq m |\mathbf{v}|_E$ for every $\mathbf{v}\in E$. Then the result follows.
\end{proof}

This gives that, for Finsler manifolds, the notion of local Lipschitz continuity previously given in the 1st definition coincides with the usual notion of local Lipschitz continuity for the Finsler metric.

\subsection*{Locally Lipschitz continuous maps: 2nd definition}
Let $X$ and $Y$  be two $C^1$ Finsler manifolds. We say that a map $f:X\rightarrow Y$ is locally Lipschitz  at $x\in X$ if
there exists an open set $W$ at $x$ and $\kappa>0$ such that for all $u,u'\in W$:
\begin{equation}\label{lipcondition2}
d_Y(f(u),f(u'))\leq \kappa d_X(u,u').
\end{equation}

\begin{remark}
Both definitions of locally Lipschitz continuous map are equivalent for mappings between Finsler manifolds. {\rm Indeed, let $X$ and $Y$  be two $C^1$ Finsler manifolds modeled on $E$ and $F$, respectively, and let $f:X\rightarrow Y$ be locally Lipschitz  at $x\in X$ according to the 1st definition. Then
there are charts $(W,\varphi)$ at $x$, $(V,\psi)$ at $f(x)$ such that $f(W)\subset V$ and the map $\mathbf{f} = \psi \circ f \circ \varphi^{-1}:\varphi(W)\rightarrow \psi(V)$ is Lipschitz  on $\varphi(W)$, namely for all $\mathbf{u},\mathbf{u}'\in\varphi(W)$ and some $\kappa>0$ \eqref{lipcondition} holds.  By Lemma \ref{lemaequivalencia2}, there is an open neighborhood $U_x\subset W$ such that, if $u,u'\in U_x$, $\mathbf{u}=\varphi(u)$ and $\mathbf{u}'=\varphi(u')$, without loss generality we can deduce that  for some $m>0$ and $n>0$:
$$ d_Y(f(u),f(u')) \leq \kappa mn(1+\varepsilon)^2 d_X(u,u')$$
Therefore $f$ is locally Lipschitz at $x$ according to 2nd definition. The converse implication is analogous.
 }
\end{remark}


\section{\bfseries\sffamily\large Metric regularity on Finsler manifolds}

Let $X$ and $Y$  be two $C^1$-smooth Finsler manifolds modeled on Banach spaces. A map  $f:X\rightarrow Y$  is said to be  {\it metrically regular} around $x$ with modulus $\mu>0$ if there exist neighborhoods $W_{x}$ and $V_{f(x)}$  such that
\begin{equation}\label{metricregularity}
d_X(u,f^{-1}(y))\leq \mu~ d_Y(y,f(u))
\end{equation}
for all $u\in W_{x}$ and  $y\in V_{f(x)}$. The infimum of such moduli $\mu$ is called {\it rate of metric regularity} and, as usual, will be denoted by {$\reg f(x)$}. If no such neighborhoods and modulus exist we set $\reg f(x)=\infty$.  As Ioffe points out in \cite{ioffe2015}: ``the very fact  that $f$ is regular near certain point is independent of the choice of specific metrics. Thus, although the definitions explicitly use metrics the regularity is a topological property''.

It is well known (see \cite{ioffe2015}) that $f$ is metrically regular at $x$ if and only if  $f$ is {\it open with linear rate around}  $x$, namely if there exist a neighborhood $W_{x}$ and a constant $\alpha>0$ such that for every $u\in W_{x}$ and $r>0$ with $B_X(u;r)\subset W_{x}$:
\begin{equation}
\label{graves}B_{Y}(f(u);\alpha r)\subset f(B_{X}(u;r))
\end{equation}
The supremum of all  $\alpha>0$ such that for some neighborhood $W_{x}$  \eqref{graves} holds $\mbox{ for all } u\in W_{x} \mbox{ and all } r>0 \mbox{ with } B_X(u;r)\subset W_{x}$ is the {\it rate of surjection} of $f$ near $x$ denoted by {$\cov f(x)$}.

A  mapping $f$ is open with linear rate around $x$ if and only if the set valued  mapping $f^{-1}:Y\to 2^{X}$,  defined by $f^{-1}(y)=\{x\in X: y=f(x)\}$ for $y\in Y$,  has the so-called  pseudo-Lipschitz property  near $(f(x),x)$. Recall, a set-valued mapping $g:Y\to 2^{X}$  has the {\it pseudo-Lipchitz property} around $(y,x)$ if there exist neighborhoods $V_{y}$ and $W_{x}$ and a number $\mu>0$ such that
$d_X(u,g(z))\leq \mu d_Y(w,z)$ provided $w,z\in V_{y}$, $u\in W_{x}$ and $u\in g(w)$. The infimum of such $\mu$ ---denoted by {$\lip g(y|x)$}--- is called the {\it Lipschitz rate} of $g$ near $(y,x)$. If no such neighborhoods and modulus exist we set $\lip g(y|x)=\infty$.

In summary, as can be seen in  \cite{ioffe2015}, we have that $f$ is metrically regular  around $x$ if and only if it is  open with linear rate around $x$ if and only if  $f^{-1}$  has the pseudo-Lipschitz  property  near $(f(x),x)$. Moreover, under the convention
$0\cdot\infty =1$,
 \begin{equation}\label{igualdadcovlip}
 \cov f(x)\cdot \reg f(x) = 1 \hspace{1cm} \reg f(x)=\lip f^{-1}(f(x)|x).
 \end{equation}

\

\begin{remark} \label{equalindex}
{\bf[Linear  case] }{\rm
 Let $\mathbf{T}\in \mathcal{L}(E,F)$ be a continuous linear operator between Banach spaces. Then the following statements are equivalent:
 \begin{enumerate}
\item[(a)]  $\mathbf{T}$  is onto
\item[(b)]  The  {\it Banach constant of $\mathbf{T}$}  is positive i.e. $$C(\mathbf{T}):=\inf_{|\mathbf{y}^*|=1}|\mathbf{T}^*{\mathbf{y}^*}|>0$$
\item[(c)]  The index $I(\mathbf{T}):=\sup_{|\mathbf{y}|=1}\left\{\inf_{\mathbf{T}\mathbf{x}		=\mathbf{y}}|\mathbf{x}|\right\}<\infty.$
\end{enumerate}
Of course,  for all $\mathbf{x}\in E$, $\cov \mathbf{T}(\mathbf{x})=C(\mathbf{T})$ and
$\reg \mathbf{T}(\mathbf{x}) = I(\mathbf{T})$. Furthermore, it is easily seen that $\mathbf{T}$ is one to one and $\mathbf{T}(E)$ is closed, if and only if,
the {\it dual Banach constant of $\mathbf{T}$} is positive i.e.  $$C^*(\mathbf{T}):=\inf_{|\mathbf{u}|=1}|\mathbf{T}{\mathbf{u}}|>0.$$
Furthermore if $\mathbf{T}$ is a linear isomorphism  then:
\begin{equation}\label{igualdadisomorphism}
  C^*(\mathbf{T})=C(\mathbf{T})=\|\mathbf{T}^{-1}\|^{-1}
\end{equation}}
\end{remark}

\

\begin{corollary}\label{regularitycharts}
Let $X$ be a $C^1$ Finsler manifold modeled on a Banach space $E$. Let  $x\in X$ and $(W,\varphi)$ be a chart at $x$. Then $\varphi: W \to E$ is metrically regular  at $x$ and $\varphi^{-1}: \varphi (W) \to X$ is metrically regular at $\mathbf{x}=\varphi(x)$. Moreover, with the equivalent norm $\|\cdot\|_{x,\varphi}:=\|d\varphi^{-1}(\varphi(x)(\cdot))\|_x$ on $E$, we have:
\begin{center}
$\cov \varphi(x)= 1 =\cov\varphi^{-1}(\mathbf{x}).$
 \end{center}
\end{corollary}

\begin{proof}
Let $\varepsilon>0$. By Lemma  \ref{lemaequivalencia1},  there is  an open neighborhood $U_x\subset W$  such that such that for every $u,u'\in U_x$:
\begin{equation}\label{ecuaciontemporal}
\frac{1}{1+\varepsilon}\|\varphi(u)-\varphi(u')\|_{x,\varphi} \leq d_X(u,u')\leq (1+\varepsilon) \|\varphi(u)-\varphi(u')\|_{x,\varphi},
\end{equation}
where $\|\cdot\|_{x,\varphi}$ is the equivalent norm  on $E$  given in Lemma \ref{lemaequivalencia1}. This means that $ \varphi: U_x \to \varphi (U_x)$ is $(1+\varepsilon)$-bi-Lipschitz. Therefore, if we consider the local Lipschitz constant of $\varphi$ at $x$:
$$
{{\rm lip~} \varphi(x)}:=\inf_{R>0}\sup\left\{\frac{\|\varphi (u) -\varphi (u')\|_{x,\varphi}}
{d_X(u, u')} \, : \, u, u'\in B_X(x;R) \mbox{ and } u\neq  u'\right\},
$$
we obtain that
$$
\frac{1}{1+\varepsilon} \leq {{\rm lip~} \varphi(x)} \leq 1+\varepsilon.
$$
As a consequence, we have that ${{\rm lip~} \varphi(x)}=1$. In the same way, if we denote $\mathbf{x}=\varphi(x)$, we also have that
${\rm lip~} \varphi^{-1} (\mathbf{x})=1$.

Since $\varphi: W \to \varphi(W)$ is a single-valued mapping, then it has the pseudo-Lipschitz property near the point $(x, \varphi(x))$  if and only if it is locally Lipschitz continuous  at $x$, and furthermore
\begin{equation}\label{igualdadmonovaluadalip}
{\lip \varphi(x) = \lip \varphi(x|\varphi (x))}.
\end{equation}
The same holds for $\varphi^{-1}: \varphi (W)\to W$, and in this case we have that
\begin{equation}\label{igualdadmonovaluadalip}
{\lip \varphi^{-1}(\mathbf{x})=\lip \varphi^{-1}(\mathbf{x}|\varphi^{-1}(\mathbf{x}))}.
\end{equation}
Therefore, an application of \eqref{igualdadcovlip} gives that $\cov \varphi(x)= 1 =\cov\varphi^{-1}(\mathbf{x})$.
\end{proof}

\begin{definition}{\bf [Finsler regularity index]} {\rm
Let $X$ and $Y$  be two $C^1$-smooth Finsler manifolds modeled on Banach spaces $E$ and $F$. For every  $x\in X$ and every linear operator $T\in \mathcal{L}(T_xX,T_{f(x)}Y)$, consider its Banach constant with respect to the dual norms $|\cdot |_x$ on $(T_xX)^*$ and $|\cdot |_{f(x)}$ on $(T_{f(x)}Y)^*$, respectively  given by the Finsler norms on $T_xX$ and $T_{f(x)}Y$:
$$C_{\tiny\rm Finsler}(T)=\inf_{|y^*|_{f(x)}=1}|T^*y^*|_x$$
Now let $f:X \to Y$ be a locally Lipschitz map, and let $Jf$ be a pseudo-Jacobian mapping for $f$. We define the {\it Finsler regularity index} of $Jf$ at $x$ by:
$$C(Jf(x)):=\sup_{R>0}\inf\{C_{\tiny\rm Finsler}(T): T\in {\rm co}\hspace{0.02in}Jf(B_X(x;R))\}$$
where $Jf(B_X(x;R))=\{ T \in Jf(u): u\in B_X(x;R)\}$.}
\end{definition}

\begin{proposition}\label{metricallyregularindex}
The Finsler regularity index for a pseudo-Jacobian mapping is well defined in the following  sense: {\rm for every  chart $(W,\varphi)$ at $x$ and every chart $(V,\psi)$ at $f(x)$ such that $f(W)\subset V$ we have that if $\mathbf{f} = \psi \circ f \circ \varphi^{-1}$,
$T=d\psi^{-1}(\mathbf{y})\mathbf{T}[d\varphi^{-1}(\mathbf{x})]^{-1}$, $\mathbf{y}=\psi(f(x))$, and $\mathbf{x}=\varphi(x)$ then:
\begin{equation}\label{igual-Fin}
C_{\tiny\rm Finsler}(T)= C_{\psi,\varphi}(\mathbf{T}) := \inf_{|\mathbf{y}^*|_{f(x),\psi}=1}|\mathbf{T}^*{\mathbf{y}^*}|_{x,\varphi}\end{equation}
where $|\cdot|_{x,\varphi}$ and $|\cdot|_{f(x),\psi}$ are the dual norms of
$\|\cdot\|_{x,\varphi}$  and $\|\cdot\|_{f(x),\psi}$, respectively. Moreover:
\begin{equation}\label{igualdadbanachJ}
C(Jf(x)) = C(J\mathbf{f}(\mathbf{x})):=\sup_{R>0}\inf\{C_{\psi,\varphi}(\mathbf{T}): \mathbf{T}\in {\rm co}\hspace{0.02in}J\mathbf{f}(B_E(\mathbf{x};R))\}
\end{equation}
}
\end{proposition}

\begin{proof} We have that
$$I_{\tiny\rm Finsler}(T)=\sup_{\|y\|_{f(x)}=1}\left\{\inf_{Tu=y}\|u\|_x\right\}.$$
Let us consider the changes of variables $\mathbf{w}:=[d\varphi^{-1}(\mathbf{x})]^{-1}(u)$
and $\mathbf{v}:= [d\psi^{-1}(\mathbf{y})]^{-1}(y)$. Then \eqref{igual-Fin} holds since:
$$I_{\tiny\rm Finsler}(T)=\sup_{\|\mathbf{v}\|_{f(x),\psi}=1}\left\{\inf_{\mathbf{T}\mathbf{w}=\mathbf{v}}\|\mathbf{w}\|_{x,\varphi}\right\}=I_{\psi,\phi}(\mathbf{T}).$$

Now for $R>0$ set
\begin{eqnarray*}
A_R & := & \{C_{\tiny\rm Finsler}(T): T\in {\rm co}\hspace{0.02in}Jf(B_X(x;R))\} \\
B_R & := & \{C_{\psi,\varphi}(\mathbf{T}): \mathbf{T}\in {\rm co}\hspace{0.02in}J\mathbf{f}(B_E(\mathbf{x};R))\}
\end{eqnarray*}

\

\noindent {\bf Claim 1:}  {\it There is $R_0>0$ such that for all $0<R<R_0$, $\inf B_R\leq C(Jf(x)) $}
\begin{enumerate}
\item[] Set $R_0>0$ such that $B_E(\mathbf{x}; R_0)\subset\varphi(W)$. Given $0<R<R_0$, there is $R'>0$, depending on $R$, such that
$$B_X(x;R')\subset\varphi^{-1}\left(B_E(\mathbf{x}; R)\right)\subset W.$$
Let $c\in A_{R'}$; then $c=C_{\tiny\rm Finsler}(T)$ for $T\in  {\rm co}\hspace{0.02in}Jf(B_X(x; R'))$. Note that,
since $s=\Sigma_{i=1}^n\lambda_i  d\psi^{-1}(\mathbf{y})  \mathbf{T}_i  [d\varphi^{-1}(\mathbf{x})]^{-1}= d\psi^{-1}(\mathbf{y}) (\Sigma_{i=1}^n \mathbf{T}_i ) [d\varphi^{-1}(\mathbf{x})]^{-1}$, we have that
$$
{\rm co}\hspace{0.02in}Jf(B_X(x; R')) = d\psi^{-1}(\mathbf{y}) \hspace{0.05in}{\rm co}\hspace{0.02in} J\mathbf{f}(\varphi(B_X(x; R'))) \hspace{0.05in} [d\varphi^{-1}(\mathbf{x})]^{-1},$$
and thus ${\rm co}\hspace{0.02in} J\mathbf{f}(\varphi(B_X(x; R'))) $ is contained in  ${\rm co}\hspace{0.02in} J\mathbf{f}(B_E(\mathbf{x};R)).$
Therefore we have that  $T=d\psi^{-1}(\mathbf{y})\mathbf{T}[d\varphi^{-1}(\mathbf{x})]^{-1}$ with $\mathbf{T}\in {\rm co}\hspace{0.02in} J\mathbf{f}(B_E(\mathbf{x};R))$ and  by equality  \eqref{igual-Fin} we deduce that $c=C_{\psi,\varphi}(\mathbf{T})$. Then $A_{R'}\subset B_R$, so:
$$\inf B_R \leq \inf {A_{R'}}\leq C(Jf(x)). $$
\end{enumerate}

\

\noindent Analogously, we have:

\

\noindent {\bf Claim 2:} {\it There is $R_0>0$ such that for all $0<R<R_0$, $\inf A_R\leq  C(J\mathbf{f}(\mathbf{x})) $}.

\

\noindent By Claim 1, letting $R\rightarrow 0$ we get $ C(J\mathbf{f}(\mathbf{x}))\leq C(Jf(x)) $, and by Claim 2, letting  $R\rightarrow 0$ we get $C(Jf(x))\leq C(J\mathbf{f}(\mathbf{x}))$. So we get
\eqref{igualdadbanachJ}.

\end{proof}

In order to obtain a connection between metric regularity of a function $f$ and the Finsler regularity index of a given  pseudo-Jacobian for $f$, the idea is to take advantage of the known results in the case of Banach spaces, and carefully transfer them to the context of Finsler manifolds. More precisely, we will use Theorem 3.1  and Corollary 2.18 of \cite{JaLaMa}. The bottom line is that the function $\mathbf{x} \mapsto \| \mathbf {f} (\mathbf{x})\|$ must satisfy a sort of a chain rule condition. This condition applies in the most relevant cases as we will detail later. So, we need some more extra assumptions.

\begin{definition}{\bf [Strong pseudo-Jacobians]} {\rm
Let $X$ and $Y$  be two $C^1$ manifolds modeled on Banach spaces $E$ and $F$, respectively. Let  $f:X\rightarrow Y$  be a locally Lipschitz map, and let $Jf$ be a pseudo-Jacobian mapping for $f$. We say that $Jf$ is  a {\it strong pseudo-Jacobian mapping} for  $f$ if, for each $x\in X$, there exists a chart $(W,\varphi)$ at $x$ and a chart $(V,\psi)$ at $f(x)$ satisfying the conditions of Definition \ref{PJ} and, in addition, if we denote $\mathbf{f} = \psi \circ f \circ \varphi^{-1}$,
\begin{itemize}
\item[(PJ$_4$)] For every $\mathbf{u}$ in the domain of the assignment $\mathbf{u}\mapsto J\mathbf{f}(\mathbf{u})$ given in condition (PJ$_2$), $J\mathbf{f}(\mathbf{u})\subset\mathcal{L}(E,F)$ is convex and compact in the weak operator topology (WOT).
\end{itemize}}
\end{definition}

For example, $Jf$ is a strong pseudo-Jacobian mapping if, for every point,  $J\mathbf{f}(\mathbf{u})$  is a Clarke-like generalized Jacobian of $\mathbf{f}$ at $\mathbf{u}$ in the domain of $J\mathbf{f}$. Examples of Clarke-like generalized Jacobians are the Clarke generalized gradient in the infinite-dimensional case, and the Pal\'es-Zeidan generalized Jacobians in the case where the target space is reflexive, see Example 2.5 in \cite{JaLaMa} and references therein.

\begin{definition}{\bf [Finsler manifold with smooth norm]} {\rm
We shall say that a Finsler manifold  {\it $X$ has smooth norm} if, for every $x\in X$, the norm $\Vert \cdot \Vert_x$ is Fr\'echet differentiable (away from $0$.)}
\end{definition}

Of course, a Riemannian manifold has smooth norm. Finite dimensional Finsler manifolds can always be given an equivalent smooth norm. In general, if $X$ is a paracompact smooth manifold modeled on a Banach space $E$ endowed with a Fr\'echet differentiable norm, by  Lemma 2.8 in \cite{palais},  $X$ admits a Finsler structure with smooth norm.

\begin{theorem}\label{lemmaprincipal}
Let $X$ and $Y$  be two $C^1$ Finsler manifolds modeled on Banach spaces, where $Y$ has smooth norm. Let $f:X\rightarrow Y$ be a locally Lipschitz map, and suppose that $Jf(x)$ is a strong pseudo-Jacobian mapping for $f$. Then, for each $x\in X$:
\begin{equation}\label{inequalitycovindex}\cov f(x)\geq C(Jf(x)).\end{equation}
\end{theorem}

\begin{proof}
Suppose that $X$ and $Y$ are modeled, respectively, on Banach spaces $E$ and $F$, respectively. Let $x\in X$ and assume  that $C(Jf(x))>0$, since otherwise the inequality is obvious. Let  $(W,\varphi)$ be a chart at $x$ and let $(V,\psi)$ be a chart at $f(x)$ such that $f(W)\subset V$. Set
$\mathbf{f} := \psi \circ f \circ \varphi^{-1}:\varphi(W)\rightarrow\psi(V)$.  Let us endow $E$ and $F$ with the norms $\|\cdot\|_{x,\phi}$ and $\|\cdot\|_{f(x),\psi}$ respectively. Let  $R>0$ such that  $$\beta:=\beta_R=\inf\{C_{\psi,\varphi}(\mathbf{T}): \mathbf{T}\in {\rm co}\hspace{0.02in}J\mathbf{f}(B_E(\mathbf{x};R))\}>0.$$
Since ${\rm co}\hspace{0.02in}J\mathbf{f}(B_E(\mathbf{x};R)):={\rm co}\hspace{0.02in}\{ \mathbf{T} \in J\mathbf{f}(\mathbf{u}): \mathbf{u}\in B_E(\mathbf{x};R)\}$ we have:
$$\bigcup_{\mathbf{u}\in B_R} {\rm co}\hspace{0.02in} J\mathbf{f}(\mathbf{u})\subset {\rm co}\hspace{0.02in}\left (\bigcup_{\mathbf{u}\in B_R} {\rm co}\hspace{0.02in} J\mathbf{f}(\mathbf{u})\right )={\rm co}\hspace{0.02in}\left (\bigcup_{\mathbf{u}\in B_R} J\mathbf{f}(\mathbf{u})\right)={\rm co}\hspace{0.02in}J\mathbf{f}(B_R)$$
where  $B_R:=B_E(\mathbf{x};R)$. Therefore,
$$\inf\{C_{\psi,\varphi}(\mathbf{T}): \mathbf{T}\in {\rm co}\hspace{0.02in}J\mathbf{f}(\mathbf{u}); \mathbf{u}\in B_E(\mathbf{x};R)\}\geq\beta>0.$$ Since $Y$ has smooth norm then $\|\cdot\|_{f(x)}$ is Fr\'echet differentiable. Furthermore, since $Jf$ is a strong pseudo-Jacobian mapping,  then  $J\mathbf{f}(\mathbf{u})$ is convex and compact in the weak operator topology for some $\mathbf{u}$ in a ball $B_E(\mathbf{x};R)$ small enough. By Theorem 3.1  and Corollary 2.18 in \cite{JaLaMa}, for each open ball $B_E(\mathbf{u};\delta)\subset B_E(\mathbf{x};R)$ we have:
\begin{equation}\label{metricregularitybold}
B_F(\mathbf{f(u)};\delta\beta)\subset \mathbf{f}(B_E(\mathbf{u};\delta))
\end{equation}
Now let  $\varepsilon >0$ be fixed. By Theorem \ref{regularitycharts},
there is an open neighborhood  $U_\varepsilon$ at $u$ such that for each open ball $B_X(u;\delta')\subset U_\varepsilon$:
 \begin{equation}\label{metricregularityvarphi}
B_E(\varphi(u);(1+\varepsilon)^{-1}\delta')\subset \varphi(B_X(u;\delta')),
\end{equation}
and there is an open neighborhood  $V_\varepsilon$ of $\mathbf{f(x)}$ such that for each open ball $B_F(\mathbf{f(u)};\delta'')\subset V_\varepsilon$:
 \begin{equation}\label{metricregularitypsi}
B_Y(\psi^{-1}(\mathbf{f(u)});(1+\varepsilon)^{-1}\delta'')\subset \psi^{-1}(B_F(\mathbf{f(u)};\delta'')).
\end{equation}
Let $W_x\subset U_\varepsilon$ small enough such that for all $r>0$ and $u\in W_x$ with $B_X(u;r)\subset W_x$,
$B_E(\varphi(u); (1+\varepsilon)^{-1}r)\subset B_E(\mathbf{x}; R)$, and
 $B_F(\psi(f(u)); (1+\varepsilon)^{-1}\beta r)\subset V_\varepsilon$.
\noindent Let  $u\in W_x$ and $r>0$ with $B_X(u;r)\subset W_x$.
Set $\rho:=\frac{\beta}{(1+\varepsilon)^2}$. By the metric regularity property of $\psi^{-1}$ at $\mathbf{f(u)}$ \eqref{metricregularityvarphi}, we have:
$$B_Y(\psi^{-1}(\mathbf{f}(\mathbf{u}));\rho  r)\subset \psi^{-1}(B_F(\mathbf{f(u)}; \rho(1+\varepsilon)r)).$$
Therefore:
$$\psi(B_Y(\psi^{-1}\mathbf{f}(\mathbf{u});\rho r))\subset B_F(\mathbf{f(u)};  \rho(1+\varepsilon) r).$$
Since $\rho(1+\varepsilon) = \beta(1+\varepsilon)^{-1}$, by \eqref{metricregularitybold} with $\delta:=(1+\varepsilon)^{-1}r$ we get:
$$\psi(B_Y(\psi^{-1}(\mathbf{f}(\mathbf{u}));\rho r))\subset \mathbf{f}\left( B_E\left(\mathbf{u}; \delta  \right)\right).$$
So,
$$B_Y((\psi^{-1}\circ\mathbf{f})(\mathbf{u});\rho r)\subset (\psi^{-1}\circ \mathbf{f})\left( B_E\left(\varphi(u); \delta  \right)\right).$$
Now, by the metric regularity property of $\varphi$ at $x$ \eqref{metricregularitypsi} and since $\mathbf{u}=\varphi(u)$ we conclude:
$$B_Y((\psi^{-1}\circ\mathbf{f}\circ\varphi)(u);\rho r)\subset (\psi^{-1}\circ \mathbf{f}\circ \varphi)(B_X(u; r)).$$
Finally, $B_Y(f(u);\rho r)\subset f(B_X(u;r)),$ and then $\cov f(x)\geq \frac{\beta}{(1+\varepsilon)^2}.$ Taking the supremum over $R>0$ we get:
$$\cov f(x)\geq \frac{C(Jf(x))}{(1+\varepsilon)^2}.$$
Since $\varepsilon>0$ was chosen arbitrary, we get the desired inequality.
\end{proof}


\section{\bfseries\sffamily\large Inverse Mapping Theorem}

\begin{definition}{\rm {\bf [Finsler local-inyectivity index]}
Let $X$ and $Y$  be two $C^1$ Finsler manifolds modeled on Banach spaces. Let $f:X\rightarrow Y$ be a locally Lipschitz map, and suppose that $Jf(x)$ is a pseudo-Jacobian mapping for $f$. Let $x\in X$ and, for every linear operator $T\in \mathcal{L}(T_xX,T_{f(x)}Y)$, consider the dual  Banach constant  given by the Finsler norms on $T_xX$ and $T_{f(x)}Y$, respectively:
 $$C^*_{\tiny\rm Finsler}(T)=\inf_{|u|_{x}=1}|Tu|_{f(x)}.$$
We define the {\it Finsler local-inyectivity index} for a  $Jf$ at $x$ by:
$$C^*(Jf(x)):=\sup_{R>0}\inf\{C^*_{\tiny\rm Finsler}(T): T\in {\rm co}\hspace{0.02in}Jf(B_X(x;R))\}$$}
\end{definition}

In a similar way to Proposition \ref{metricallyregularindex} we can prove that:

\begin{proposition}
The Finsler local-inyectivity index for pseudo-Jacobian mapping is well defined in the following  sense: {\rm for every  chart $(W,\varphi)$ at $x$ and every chart $(V,\psi)$ at $f(x)$ such that $f(W)\subset V$ we have that if $\mathbf{f} = \psi \circ f \circ \varphi^{-1}$,
$T=d\psi^{-1}(\mathbf{y})\mathbf{T}[d\varphi^{-1}(\mathbf{x})]^{-1}$, and $\mathbf{x}=\varphi(x)$ then:
\begin{equation}\label{igualdadbanach}C^*_{\tiny\rm Finsler}(T)= C^*_{\psi,\varphi}(\mathbf{T}) := \inf_{\|\mathbf{u}\|_{x,\varphi}=1}\|\mathbf{T}{\mathbf{u}}\|_{f(x),\psi}.\end{equation}
As expected:
\begin{equation}\label{igualdadbanachJdual}
C^*(Jf(x)) = C^*(J\mathbf{f}(\mathbf{x})):=\sup_{R>0}\inf\{C^*_{\psi,\varphi}(\mathbf{T}): \mathbf{T}\in {\rm co}\hspace{0.02in}J\mathbf{f}(B_E(\mathbf{x};R))\}
\end{equation}}
\end{proposition}

\begin{theorem}{\bf [Local injectivity]}\label{localinjectivity}
Let $X$ and $Y$  be two $C^1$ Finsler manifolds modeled on Banach spaces. Let $f:X\rightarrow Y$ be a locally Lipschitz map, and suppose that $Jf(x)$ is a pseudo-Jacobian mapping for $f$. Let  $x\in X$ and suppose that $C^*(Jf(x))>\alpha>0$. Then $f$ is locally injective at $x$, an more precisely there is an open neighborhood $U$ of $x$ such that for all $u,u'\in U$:
$$d_Y(f(u),f(u'))\geq\alpha\hspace{0.01in} d_X(u,u').$$
\end{theorem}

\begin{proof}
Denote by $E$ and $F$, respectively, the model spaces of $X$ and $Y$, respectively. Let  $(W,\varphi)$ be a chart at $x$ and let $(V,\psi)$ be a chart at $f(x)$ such that $f(W)\subset V$. Given $\varepsilon >0$, let $U_x$ be an open neighborhood with $U_x\subset W$ and such that for all $u,u'\in U_x$:
$$
\frac{1}{(1+\varepsilon)}\|\varphi(u)-\varphi(u')\|_{x,\varphi} \leq d_X(u,u')\leq (1+\varepsilon) \|\varphi(u)-\varphi(u')\|_{x,\varphi}.
$$
We may also assume that $ f(U_x)\subset O_{f(x)}\subset f(W)$ for some $O_{f(x)}$ such that  for all $y,y'\in O_{f(x)}$:
$$
\frac{1}{(1+\varepsilon)}\|\psi(y)-\psi(y')\|_{f(x),\psi} \leq d_Y(y,y')\leq (1+\varepsilon) \|\psi(y)-\psi(y')\|_{f(x),\psi}.
$$
Let $\mathbf{f} = \psi \circ f \circ \varphi^{-1}:\varphi(U_x)\rightarrow\psi(O_{f(x)})$. We have that there exists $R>0$ such that $B_E(\mathbf{x};R)\subset \varphi(U_x)$  and, for all $\mathbf{T}\in {\rm co}\hspace{0.02in}J\mathbf{f}(B_E(\mathbf{x};R))$ we have that
$$C^*_{\psi,\varphi}(\mathbf{T})\geq\beta>\alpha>0 $$
for some $\beta$. Take $\mathbf{u},\mathbf{u'}\in B_E(\mathbf{x};R)$ and pick a $0<r<1$. Let us endow $E$ and $F$ with the norms $\|\cdot\|_{x,\phi}$ and $\|\cdot\|_{f(x),\psi}$ respectively. By Theorem 2.7 of \cite{JaLaMa} there exists an operator $\mathbf{T}_0\in \mbox{\rm co}(J\mathbf{f}[\mathbf{u},\mathbf{u'}])$ such that
$$\|\mathbf{f}(\mathbf{u})-\mathbf{f}(\mathbf{u'})-\mathbf{T}_0(\mathbf{u}-\mathbf{u'})\|_{f(x),\psi}\leq r\beta\|\mathbf{u}-\mathbf{u'}\|_{x,\varphi}.$$ Since $[\mathbf{u},\mathbf{u'}]\subset B_E(\mathbf{x};R)\subset \varphi(W)$ we have that $C^*_{\psi,\varphi}(\mathbf{T}_0)>\alpha>0$. Then as in proof of Lemma 3.8 in \cite{JaLaMa} we get
$$\|\mathbf{f}(\mathbf{u})-\mathbf{f}(\mathbf{u'})\|_{f(x),\psi}\leq (1-r)\beta\|\mathbf{u}-\mathbf{u'}\|_{x,\varphi}.$$
Since $r$ was chosen arbitrarily, we get that for all $\mathbf{u},\mathbf{u'}\in  B_E(\mathbf{x};R)$:
$$\|\mathbf{f}(\mathbf{u})-\mathbf{f}(\mathbf{u'})\|_{f(x),\psi}\geq\beta\|\mathbf{u}-\mathbf{u'}\|_{x,\varphi}.$$
Therefore, for all $u,u'\in U_\varepsilon:=\varphi^{-1}(B_E(\mathbf{x};R))$
 $$\| \psi \circ f \circ \varphi^{-1}(\varphi(u))-\psi \circ f \circ \varphi^{-1}(\varphi(u'))\|_{f(x),\psi}\geq\beta\|\varphi(u)-\varphi(u')\|_{x,\varphi}.$$
Then:
$$\| \psi( f (u))-\psi( f (u'))\|_{f(x),\psi}\geq \beta\|\varphi(u)-\varphi(u')\|_{x,\varphi}.$$
By above inequalities we get:
$$d_Y(f(u),f(u'))\geq\frac{\beta}{(1+\varepsilon)^2}d_X(u,u').$$
It is enough to choose $\varepsilon>0$ such that $\alpha<\frac{\beta}{(1+\varepsilon)^2}<\beta$.
\end{proof}

\begin{definition}
{\rm Let  $f:X\rightarrow Y$ be a locally Lipschitz map between $C^1$ Finsler manifolds, and let $Jf$ be a  pseudo-Jacobian mapping for $f$.  We shall say  that $f$ is {\it $Jf$-regular} at a point  $x \in X$ if:
 \begin{enumerate}
 \item   $C(Jf(x))>0$, and
 \item   there exists is $R>0$ such that every $T$ in ${\rm co}\hspace{0.02in}Jf(B_X(x;R))$ is a linear isomorphism.
\end{enumerate}
Note that if condition $(2)$ holds, we have that
\begin{equation}\label{inversioncondition}
C(Jf(x))=C^*(Jf(x)).
\end{equation}}
\end{definition}

Combining Lemma \ref{lemmaprincipal} and Theorem \ref{localinjectivity} we obtain at once the following local inversion result.

\begin{theorem}{\bf [Inverse Mapping Theorem]}\label{inversemapping}
Let $f:X\rightarrow Y$ be locally Lipschitz map between  $C^1$ Finsler manifolds, where $Y$ has smooth norm, and let  $Jf$ be a strong pseudo-Jacobian mapping for $f$. Suppose that $f$ is  $Jf$-regular at a point  $x \in X$, with  $C(Jf(x))> \alpha >0$. Then there is an open neighborhood $U$ at $x$ such that:
\begin{enumerate}
\item[$(1)$] $f_x:= f|_U:U\rightarrow f(U)$ is a homeomorphism with $$d_Y(f_x(u),f_x(u'))\geq\alpha d_X(u,u')\mbox{ for all } u,u'\in U.$$
\item[$(2)$] $f_x$ is an open map with linear rate at every point of $U$ and also:  $$C^*(Jf(u))=C(Jf(u))\geq\alpha \mbox{ for all } u\in U.$$
\item[$(3)$]  $f_x^{-1}$ is Lipschitz continuous and $\lip f_x^{-1}(f(u)) \leq \alpha^{-1}$ for all $u\in U$.
\end{enumerate}
\end{theorem}


\section{ \bfseries\sffamily\large Global inverse theorems}

The Ehresmann Theorem asserts that a proper submersion $f:X\rightarrow Y$ between finite dimensional manifolds, where $X$ paracompact and $Y$ connected, is a locally trivial fiber bundle. Recall that a map $f:X\rightarrow Y$ between manifolds is a {\it submersion} is $f$ is differentiable and such that $df(x)$ is surjective for all $x\in X$. On the other hand, $f$ is said to be a {\it proper map} provided $f^{-1}(K)$ is compact in $X$ whenever $K$ is compact in $Y$. In a remarkable work \cite{rabier}, Rabier extends Ehresmann theorem to the framework of infinite-dimensional Finsler manifolds  via the notion of  ``strong submersions'', which is a generalization of proper submersions (see Theorem 4.1 in \cite{rabier}). The notion of strong submersion is closely related to the Palais-Smale condition for a non-linear functional $f$,  setting  down that there should be no sequence $\{x_n\}$ in  $X$ such that $\{f(x_n)\}$ converges and $\|df(x_n)\|$ tends to $0$. In this form, it can be generalized to mappings between Finsler manifolds via the Banach constant. So, acording to Definition 3.2 in \cite{rabier}, and using the our notation as in Remark \ref{equalindex}, a $C^1$ mapping $f:X\rightarrow Y$ between $C^1$ Finsler manifolds is a {\it strong submersion} is there is no sequence $\{x_n\}$ in $X$ with $f(x_n)\rightarrow y\in Y$ and $C(df(x_n))\rightarrow 0$. In  our context it is natural to set up the following definition:

\begin{definition}{\bf [Strong submersion]} {\rm Let  $f:X\rightarrow Y$ be a locally Lipschitz  map between $C^1$ Finsler manifolds, and let $Jf$ be a pseudo-Jacobian mapping for $f$. We shall say that $f$ is a {\it strong submersion} if  there is no sequence $\{x_n\}$ in $X$ with $f(x_n)\rightarrow y\in Y$ and $C(Jf(x_n))\rightarrow 0$.}
\end{definition}

When $f$ is a local diffeomorphism, the fibers of a fiber bundle are discrete, so Theorem 4.1 in \cite{rabier} yields to Corollary 4.2 in the same reference: If {\it $f$ is both a local diffeomorphism and a strong submersion between Finsler manifolds then it is a covering map}.  We present a generalization of this result in a non-smooth setting:

\begin{theorem}{\bf [Covering maps I]}\label{covering1}
Let $X$ and $Y$  be two $C^1$ Finsler manifolds, where $X$ is complete and $Y$ has smooth norm, and let $f:X\rightarrow Y$ be a locally Lipschitz  map. Suppose that $Jf$ is a strong pseudo-Jacobian for $f$, which is  is $Jf$-regular  at every point  $x\in X$.  If $f $ is a strong submersion then $f$ is a covering map, the set-valued inverse map $f^{-1}$ has the  pseudo-Lipchitz property around $(y,x)$ for every  $y\in Y$ and $x\in f^{-1}(y)$, and  $\lip f^{-1}(y|x)\leq {C(Jf(x))}^{-1}.$
\end{theorem}

\begin{proof}
In order to apply Theorem  5.2 of  \cite{GutuJaramillo} we need to verify the following check list:
\begin{enumerate}
\item[$(i)$]  {\it $Y$ is locally $\mathcal R$-contractible  \cite[p. 78]{GutuJaramillo}}: this is a fairly general class of metric spaces with nice local structure.  In particular, Example 2.4 in \cite{GutuJaramillo} states that every Finsler manifold is a locally $\mathcal R$-contractible space.
\item[$(ii)$] {\it $f$ is a local quasi-isometric map}: this means that for every $x\in X$ there is an open neighborhood $U$ of $x$, and constants $0<m\leq M$ such that:
$$m\leq\inf_{u\in U}D_u^-f\leq\sup_{u\in U}D_u^+f\leq M$$
where $$D_u^-f=\liminf_{w\rightarrow u}\frac{d_Y(f(w),f(u))}{d_X(w,u)}\mbox{ and }D_u^+f=\limsup_{w\rightarrow u}\frac{d_Y(f(w),f(u))}{d_X(w,u)}.$$
This follows from our Inverse Mapping Theorem \ref{inversemapping} and the fact that $f$ is locally Lipschitz.
\item[$(iii)$] {\it $f$ has the continuation property for rectifiable paths}: this means that for every rectifiable path $p:[0,1]\rightarrow Y$ in $Y$, every $b\in (0,1]$ and every $q:[0,b)\rightarrow X$ such that $f\circ q = p$ over $[0,b)$, there exists $\beta>0$ such that $\inf\{D_x^-f: x \mbox{ in the image of } q\}\geq\beta.$ Indeed, since $f$ is a strong submersion, then for all $y\in Y$ there is some $\beta_y>0$ and a neighborhood $V$ of $y$ such that $C(Jf(x))\geq\beta_y$ for all $x\in f^{-1}(V)$. Since the image of $p$ is a compact set in $Y$, by a standard compacteness argument there is $\beta>0$ such that $\inf\{C(Jf(x)): x \mbox{ in the image of } q\}\geq\beta.$
Now, for each $x\in X$ and each $0<\alpha < C(Jf(x))$, by our Inverse Mapping Theorem (Theorem \ref{inversemapping}) $D_x^-f\geq\alpha$. Therefore, $D_x^-f\geq C(Jf(x))$,  and so $f$ has the continuation property for rectifiable paths.
\end{enumerate}
On the other hand note that, in fact,  the requirement ``$Y$ complete'' is not necessary in the proof of Theorem  5.2 of \cite{GutuJaramillo}. Thus, from Theorem  5.2 of \cite{GutuJaramillo} we obtain that $f$ is a covering map. Let $y\in Y$ and $x\in f^{-1}(y)$. By \eqref{igualdadcovlip} $\cov f(x)^{-1}=\reg f(x)=\lip f^{-1}(y|x)$. By \eqref{lemmaprincipal} we have that $\lip f^{-1}(y|x)\leq {C(Jf(x))}^{-1}.$
\end{proof}

\noindent Motivated by the ``weighted'' Palais-Smale condition ---such as Cerami condition \cite{schechter}---  and the Hadamard theorem for local diffeomorphism between Banach spaces \cite{hadamard, Plastock, john, katriel, gutu2015} via the integral condition:
$$\int_0^\infty \inf_{|x|\leq \rho} C(df(x))d\rho=\infty,$$ we define below the concept of {\it weighted strong submersion}.

\begin{definition}{\bf [Weighted strong submersion]} {\rm A  {\it weight} is a nondecreasing map (not necessarily continuous)  $\omega:[0,\infty)\rightarrow (0,\infty)$ such that  $$\int_0^\infty\frac{1}{\omega(\rho)}d\rho=\infty.$$
Let  $f:X\rightarrow Y$ be a locally Lipschitz  map between $C^1$ Finsler manifolds, and let $Jf$ be a pseudo-Jacobian mapping for $f$. We shall say that $f$ is a {\it weighted strong submersion} if  there is no sequence $\{x_n\}$ in $X$ with $f(x_n)\rightarrow y\in Y$ and $C(Jf(x_n))\omega(d_X(x^*,x))\rightarrow 0$ for some $x^*\in X$ and some weight.}
\end{definition}

\begin{theorem}{\bf [Covering maps II]}\label{coveringweighted}
Let $X$ and $Y$  be two $C^1$ Finsler manifolds, where $X$ is complete and $Y$ has smooth norm, and let $f:X\rightarrow Y$ be a locally Lipschitz map. Suppose that $Jf$ is a strong pseudo-Jacobian mapping for $f$ which is  is $Jf$-regular  at every point  $x\in X$. If $f$ is a weighted strong submersion then $f$ is a covering map, the set-valued inverse map $f^{-1}$ has the  pseudo-Lipschitz property around $(y,x)$ for every  $y\in Y$ and $x\in f^{-1}(y)$ with $\lip f^{-1}(y|x)\leq {C(Jf(x))}^{-1}.$
\end{theorem}

\begin{proof}
Suppose that $f$ is a weighted strong submersion for the weight $\omega$. We argue as in the proof of  Theorem \ref{covering1}, but instead of $(iii)$, we can show with the same arguments that:{\it  $f$ has the  bounded path lifting property for rectifiable paths with respect to the weight $\omega$}. Namely, for every rectifiable path  $p:[0,1]\rightarrow Y$ in $Y$, every $b\in (0,1]$ and every $q:[0,b)\rightarrow X$ such that $f\circ q = p$ over $[0,b)$, there exists $x^*\in X$ and $\beta>0$ such that $\inf\{D_x^-f \cdot \omega(d_X(x,x^*)): x \mbox{ in the image of } q\}\geq\beta.$ Therefore,  again using Theorem  5.2 of \cite{GutuJaramillo} $f$ is a covering map. As before, by  \eqref{lemmaprincipal} we have that $\lip f^{-1}(y|x)\leq {C(Jf(x))}^{-1}$.
\end{proof}

\begin{remark}\label{remarkintegralcondition}{\rm Let $x^*$ some point fixed in $X$. Reasoning as in the proof of Lemma 4.5 in \cite{GutuJaramillo}, we can verify that there exists
 a {\it weight} $\omega$ such that $C(Jf(x))\omega(d_X(x^*,x))\geq 1$  if and only if the following integral condition holds
\begin{equation}\label{integralcondition}\int_0^\infty \inf_{d_X(x,x^*)\leq \rho} C(Jf(x))d\rho=\infty.\end{equation} In particular, {\it if a metrically regular local homeomorphism satisfies the Hadamard integral condition then it is a (weighted) strong submersion}. Actually if $f:X\rightarrow Y$ is a local diffeomorphism  between Banach spaces, $f$ is a global diffeomorphism if and only if it is a strong submersion \cite{rabier, gutuchang}. Furthermore, in this context, the Hadamard's integral condition implies coercivity, namely $\lim_{|x|\rightarrow\infty} |f(x)|\rightarrow\infty$. But in infinite-dimensional setting there are non-coercive global diffeormophisms between Banach spaces.}
\end{remark}

The following result is a direct consequence of Remark \ref{remarkintegralcondition},  Theorem \ref{coveringweighted} and the fact that if $f:X\rightarrow Y$ is a covering map with $X$ path-connected and $Y$ simply-connected then it is a homeomorphism.

\begin{corollary}{\bf [Hadamard Theorem]}\label{hadamardcondition}
Let $X$ and $Y$  be two $C^1$ Finsler manifolds, where $X$ is complete and $Y$ is simply-connected and has smooth norm.  Let $f:X\rightarrow Y$ be a locally Lipschitz map, and suppose that $Jf$ is a strong pseudo-Jacobian for $f$, which is $Jf$-regular  at every point in $X$. Consider some fixed point  $x^*\in X$ and some weight $\omega$, and assume that the following integral condition holds
$$
\int_0^\infty \inf_{d_X(x,x^*)\leq \rho} C(Jf(x))d\rho=\infty.
$$
Then $f$ is a global homeomorphism, the inverse map $f^{-1}$ is locally Lipschitz and $\lip f^{-1}(f(x))\leq \omega(d_X(x,x^*)).$
\end{corollary}

\begin{remark}{\bf[Estimate of the domain of invertibility]}  {\rm Let $X$ and $Y$  be two $C^1$ Finsler manifolds, where $X$ is complete and $Y$ has smooth norm, and  let $f:X\rightarrow Y$ be a locally Lipschitz map. Suppose that $Jf$ is a strong pseudo-Jacobian for $f$, which is $Jf$-regular  at every point in $X$, and consider  some point fixed in $x^* \in X$. For every $r>0$ set:
\begin{equation}\label{rho}\varrho(r) = \int_0^r \ \inf_{d_X(x,x^*)\leq \rho} C(Jf(x)) d\rho.\end{equation}
Then:
$$B_Y(f(x^*);\varrho(r))\subset f(B_X(x^*;r)).$$
Indeed, If $\varrho(r)=0$ the above inclusion holds trivially. Now if $\varrho(r)>0$, by  Theorem \ref{inversemapping} (our Inverse Mapping Theorem) $f$ is a metrically regular local homeomorphism. Set
  $$\mu(\rho) = \inf_{d_X(x,x^*)\leq \rho} D_x^-f.$$
  As in proof of Theorem \ref{covering1} we have that $D_x^-f \geq C(Jf(x))$. Therefore for $x$ in the  ball $B(x^*; \rho)$:
  $$D_x^-f \geq C(Jf(x)) \geq  \inf_{d_X(x,x^*)\leq \rho} C(Jf(x)).$$
  Then $\mu(\rho)\geq\frac{1}{\omega(\rho)}$ so  $\xi(r) := \int_0^r \mu(\rho) d\rho \geq \varrho(r).$ By Theorem 6 in \cite{GaGuJa} we have:
 $$ B_Y(f(x^*);\varrho(r)) \subset B_Y(f(x^*);\xi(r))\subset f(B_X(x^*;r)).$$\\}
\end{remark}

\begin{corollary}{\bf [Global metric regularity]}\label{globalregularity}
Let $X$ and $Y$  be two $C^1$ Finsler manifolds, where $X$ is complete and $Y$ has smooth norm, and let $f:X\rightarrow Y$ be a locally Lipschitz  map. suppose that $Jf$ is a strong pseudo-Jacobian for $f$, which is $Jf$-regular  at every point in $X$. Suppose that there exists $\alpha>0$ such that:
$$
C(Jf(x)) \geq \alpha >0 \hspace{1cm}\mbox{for all }x\in X.
$$
Then, for every $r>0$:
\begin{equation}\label{globalregularity}B_Y(f(x);\alpha r)\subset f(B_X(x;r)).\end{equation}
Furthermore, $f$ is  a covering map, and the set-valued inverse map $f^{-1}$ has the  pseudo-Lipchitz property around $(y,x)$ for every  $y\in Y$ and $x\in f^{-1}(y)$ with
$$\lip f^{-1}(y|x)\leq \alpha^{-1}.$$
\end{corollary}

\begin{theorem}{\bf [Characterizing metrically regular homeomorphisms]}\label{characterizationii}
Let $X$ and $Y$  be two $C^1$ Finsler manifolds, where $X$ is complete and $Y$ is simply connected and has smooth norm. Let $f:X\rightarrow Y$ be a locally Lipschitz map, and suppose that $Jf$ is a strong pseudo-Jacobian for $f$, which is $Jf$-regular  at every point in $X$.  Then, the following statements are equivalent:
\begin{enumerate}
\item[$(1)$]  $f$ is a global homeomorphism from $X$ onto $Y$.
\item[$(2)$]  for each compact set $K\subset Y$ there exists $\alpha_K>0$ such that $C(Jf(x))\geq \alpha_K$, for all $x\in f^{-1}(K)$.
\item[$(3)$]  $f$ is a strong submersion.
\item[$(4)$]  $f$ is a weighted strong submersion.
\item[$(5)$]  $f$ is a global homeomorphism from $X$ onto $Y$ with locally Lipschitz continuous inverse.
\item[$(6)$]  $f$ is a proper map.
\end{enumerate}
Furthermore, if any of these criteria are satisfied then for every $y\in Y$:
\begin{equation}\label{desigualdad}
\lip f^{-1}(y)  \leq  C(df(f^{-1}(y)))^{-1}.
\end{equation}
\end{theorem}

\begin{proof}
Suppose first that $(1)$ holds and  $K\subset Y$  is compact. Since $f^{-1}$ is a continuous map, then $f^{-1}(K)$ is compact in $X$. Since the map $x\mapsto C(J(x))$
is lower semi-continuous on $X$, it attains its minimum on $f^{-1}(K)$.  But, by regularity,   $C(Jf(x))>0$ for all $x\in X$, so there is $\alpha_K>0$ such that $C(Jf(x))\geq \alpha_K$, for all $x\in f^{-1}(K)$. Therefore (2) is fulfilled. To prove that (2) implies (3) consider a sequence $\{x_n\}$ in $X$ with $f(x_n)\rightarrow y\in Y$ and such that $C(Jf(x_n))\rightarrow 0$. Now set $K=\{f(x_n)\}_n\cup \{y\}$, and then condition (2)providesa contradiction. On the other hand, (3) implies (4) is obvious, one just needs to consider $\omega \equiv 1$. By Theorem
\ref{coveringweighted} $f$ is a covering map and since $Y$ is simply connected, in fact it is a global homeomorphism and the inverse mapping  $f^{-1}$   is a locally Lipschitz continuous mono-valued mapping such that
inequality \eqref{desigualdad} holds.  (5) implies (6) is trivial.To conclude the proof, it is enough to observe that $f$ is a local quasi-isometric map,  if $f$ is proper map then  by  Theorem 5.2 of \cite{GutuJaramillo},  $f$ is a global homeomorphism.
\end{proof}


\section{ \bfseries\sffamily\large Lipschitz perturbations}

In this Section we will study the global invertibility of ``perturbed" maps of the form $F(x)=\sigma(f(x),g(x))$, where $f:X\to Y$ is globally invertible, $g:X \to Y$ has small local Lipschitz constant and $\sigma: Y\times Y\rightarrow Y$ is a map with suitable properties, which mimics the sum of functions in the context of manifolds.

\begin{theorem}{\bf [Global inversion under Lipschitz perturbation]} \label{lipschitzperturbation}
Let $X$ and $Y$  be two $C^1$ complete Finsler manifolds, where $Y$ is simply connected and has smooth norm, and let  $f,g:X\rightarrow Y$ be locally Lipschitz maps. Suppose that $Jf$ is a strong pseudo-Jacobian for $f$, which is $Jf$-regular at every point in $X$, and there exist a point $x^*\in X$ and a weight $\omega$  such that:
\begin{equation}\label{ecuacionlipschitzperturbationversion2}
C(Jf(x)) - \lip g(x) \geq \frac{1}{\omega(d_X(x,x^*))}
\end{equation}
Let $\sigma: Y\times Y\rightarrow Y$ be a map such that:
\begin{enumerate}
\item[$\cdot$]  $\sigma(y,z)=\sigma(z,y)$   for all $y,z\in Y$ ({\it symmetry});
\item[$\cdot$]  for every $y,z\in Y$ there is a  neighborhood $W$ of $y$, depending on $y$ and $z$, such that $d_Y( \sigma(y_1,z), \sigma(y_2,z)) =  d_Y(y_1,y_2)$ for all $y_1,y_2\in W$ ({\it local isometry}).
\end{enumerate}
For every $r>0$ set \begin{equation}\label{varrho}\varrho(r) = \int_0^r \frac{1}{\omega(\rho)}d\rho.\end{equation}
Then the map $F:X \to Y$ defined by  $F(x)=\sigma(f(x),g(x))$ satisfies:
\begin{enumerate}
\item[$(1)$]  $F$ is a global homeomorphism from $X$ onto $Y$ with locally Lipschitz inverse $F^{-1}$ such that for each $y\in Y$,
\begin{equation}\label{inequalitylip}\lip F^{-1}(y)\leq \omega(d_X(x^*, F^{-1}(y))).\end{equation}
\item[$(2)$] For every $r>0$, $B_Y(F(x^*);\varrho(r))\subset F(B_X(x^*;r)).$
\end{enumerate}
\end{theorem}

\begin{proof}
\mbox{}\\
{\it Local inversion.} First we are going to prove that $F$ is a local homeomorphism metrically regular at every point $x\in X$. Let $x\in X$ be fixed. Combinig inequality \eqref{ecuacionlipschitzperturbationversion2} with Lemma \ref{lemmaprincipal} and the fact that
$\reg f(x) = (\cov f(x))^{-1}$, we have that:
$$\cov f(x)\geq C(Jf(x)) >  \lip g(x).$$ So we are under conditions of Theorem 3.8 in  \cite{durea}. Then we have:
\begin{equation}\label{regsuma}
0< \reg F(x) = \reg \sigma(f,g)(x)\leq\frac{\reg f(x)}{1-\reg f(x) \lip g(x)}.
\end{equation}
Therefore $F$ is metrically regular around $x$.    Since $\sigma$ is a local isometry, for $u,u'$ near $x$:
\begin{eqnarray*}
 d_Y( f(u) , f(u') ) & =&  d_Y( \sigma (f(u),g(u)), \sigma (f(u'),g(u))) \\
                           & \leq &  d_Y( \sigma (f(u),g(u)),  F(u')) +  d_Y( F(u'),  \sigma(f(u'),g(u)))\\
                            &=&  d_Y( F(u),F(u') ) + d_Y( g(u'),g(u) )\\
\end{eqnarray*}
Set $\kappa :=\reg f(x)$ and $\mu := \lip g(x)$. By  Theorem \ref{inversemapping} applied to $f$ (our Inverse Mapping Theorem) there is a neighborhood $U$ of $x$ such that for every $u,u'\in U$:
\begin{eqnarray*}
d_Y(F(u),F(u')) &\geq& d_Y(f(u), f(u'))-d_Y(g(u), g(u'))\\
                         &\geq& \frac{1}{\kappa} d_X(u, u')-d_Y(g(u), g(u'))
\end{eqnarray*}
Now let $\delta > 0$ such that $0<\mu < \mu + \delta < \frac{1}{\kappa}$. Then there is $R>0$ such that for all $u,u'\in B_X(x; R)$, $u\neq u'$:
$$\frac{d_Y(g(u),g(u'))}{d_X(u,u')}\leq \mu + \delta.$$
Therefore, for every  $u,u' \in B_X(x; R)\cap U$we have:
\begin{eqnarray*}
d_Y(F(u),F(u'))  &\geq& \frac{1}{\kappa} d_X(u, u')- (\mu+\delta)d_X(u, u')\\
                          &=& \left(\frac{1}{\kappa} - (\mu+\delta)\right) d_X(u, u')
\end{eqnarray*}
with $\frac{1}{\kappa} - (\mu+\delta)>0$. As a consequence, we obtain that  $F$ is locally injective. Then $F$ is a local homeomorphism,  metrically regular at every $x\in X$.\\

\noindent {\it Global properties.} For every fixed $x\in X$,  by the previous  inequality  we have that for every $\delta>0$,
$D_x^-F\geq\frac{1}{\kappa}-\mu-\delta$. Therefore, from \eqref{ecuacionlipschitzperturbationversion2} we get the global condition:
$$D_x^-F\geq \frac{1}{\kappa}-\mu = C(Jf(x))-\lip g(x)\geq \frac{1}{\omega(d_X(x,x^*))}.$$
As before, by Theorem  5.2 of \cite{GutuJaramillo} $f$ is a covering map, and since $Y$ is simply connected, $F$ is actually a global homeomorphism, metrically regular at every point with with locally Lipscthitz  continuous inverse $F^{-1}$ such that for all $y\in Y$:
\begin{equation}\label{inequalitychida}
\lip F^{-1}(y)\leq \omega(d_X(x^*, F^{-1}(y))).
\end{equation}
On the other hand, set $\mu(\rho) := \inf_{d_X(x,x^*)\leq \rho} D_x^-F$. Therefore for $x$ such that $d_X(x,x^*)\leq\rho$:
  $$D_x^-F\geq C(Jf(x))- \lip g(x) \geq \frac{1}{\omega(d_X(x,x^*))}\geq\frac{1}{\omega(\rho)}.$$
  Then $\mu(\rho)\geq\frac{1}{\omega(\rho)}$ so  $\xi(r) := \int_0^r \mu(\rho) d\rho \geq \varrho(r).$ Finally, by Theorem 6 in \cite{GaGuJa}:
 $$ B_Y(F(x^*);\varrho(r)) \subset B_Y(F(x^*);\xi(r))\subset F(B_X(x^*;r)).$$
\end{proof}

\begin{remark}
{\rm Theorem \ref{lipschitzperturbation} can be made more general, by considering a map $\sigma: Y\times Y\rightarrow Y$  with the symmetric property, but instead of being a local isometry, with the property that for every $y,z\in Y$ and every $\varepsilon>0$ there is a  neighborhood $W$ of $y$, depending on $y$ and $z$, such that for all $y_1,y_2\in W$:
$$\frac{1}{1+\varepsilon} d_Y(y_1,y_2) \leq  d_Y( \sigma(y_1,z), \sigma(y_2,z)) \leq (1+\varepsilon) d_Y(y_1,y_2).$$
Under this assumption,  Theorem 3.8 in   \cite{durea} also guarantees the metric regularity of $F=\sigma(f,g)$ at every point $x\in X$. Note  that for every $x\in X$ and $\varepsilon>0$ there is a neighborhood of $x$ such that for every $u,u'$ therein
$$d_Y(f(u),f(u'))\leq (1+\varepsilon) d_Y( F(u),F(u') ) + (1+\varepsilon)^2 d_Y( g(u'),g(u) ).$$
Therefore, reasoning as in the proof of Theorem  \ref{lipschitzperturbation}, for $u,u'$ near $x$,  we get the inequality:
$$
d_Y(F(u),F(u')) \geq  \left(\frac{1}{\kappa(1+\varepsilon)} - (\mu+\delta)(1+\varepsilon)\right) d_X(u, u').
$$ So, $F$ is locally injective. The rest of the proof is similar.}
\end{remark}

\begin{corollary}
Let $X$  be a $C^1$ complete Finsler manifold, let $E$ be a Banach space with smooth norm, and let  $f,g:X\rightarrow E$ be locally Lipschitz maps.  Suppose that $Jf$ is a strong pseudo-Jacobian for $f$, which is $Jf$-regular at every point of $X$, and there exist a point $x^*\in X$ and a weight $\omega$such that inequality
\eqref{ecuacionlipschitzperturbationversion2} holds. For every $r>0$ set $\varrho(r)$ as in \eqref{varrho}. Then the map $$F = f + g$$ is a global homeomorphism from $X$ onto $E$. Furthermore,  and for every $r>0$, we have that $B_E(F(x^*);\varrho(r))\subset F(B_X(x^*;r))$ and the inverse map $F^{-1}$ is globally ${\omega(r)}$-Lipschitz  on $B_E(F(x^*);\varrho(r))$.
\end{corollary}

\begin{proof}
Set $\sigma(y,z) := y+z$ on $E\times E$, and apply  Theorem \ref{lipschitzperturbation}. Given $r>0$, by inclusion given in item (2) in Theorem \ref{lipschitzperturbation} we see that $B_E(F(x^*);\varrho(r))\subset F(B_X(x^*;r))$. Then for every $y\in B_E(F(x^*);\varrho(r))$ we have that $d_X(x^*, F^{-1}(y)<r$ and from inequality (1) in Theorem \ref{lipschitzperturbation} we obtain that
$$
\lip F^{-1}(y)\leq \omega(d_X(F^{-1}(y),x^*))\leq \omega (r).
$$
That is, $F^{-1}$ is locally $\omega (r)$-Lipschitz on $B_E(F(x^*);\varrho(r))$. From the convexity of the ball, we deduce that, in fact,  $F^{-1}$ is globally ${\omega(r)}$-Lipschitz  on $B_E(F(x^*);\varrho(r))$.

\end{proof}

\begin{corollary}{\bf [Perturbation of the identity]}\label{perturbation}
Let $E$  be a Banach space with smooth norm $|\cdot|$, and let  $g:E\rightarrow E$ be a locally Lipschitz map.  Suppose that there exists a weight $\omega$ such that,  for every $x\in E$:
$$\lip g(x) \leq 1-\omega(|x|)^{-1} < 1.$$
 Then the map $I+g$ is a global homeomorphism of $E$ onto itself, and   for every $r>0$, the inverse map $(I+g)^{-1}$ is globally ${\omega(r)}$-Lipschitz continuous on $B_E(g(0);\varrho(r))$.
\end{corollary}

Note that if $\omega$ in the Corollary \ref{perturbation} is constant then we have the classical theorem of perturbation of the identity.

\end{document}